\documentclass{article}
\usepackage{graphicx} % Required for inserting images

\usepackage{amsmath}
\usepackage{amsfonts}
\usepackage{amsthm}
\usepackage{amssymb}

\usepackage{stmaryrd}

\usepackage[abbrev,msc-links, alphabetic, backrefs]{amsrefs}
\usepackage{amsrefs}
\usepackage{doi}

\renewcommand{\PrintDOI}[1]{\doi{#1}}

\BibSpec{arXiv}{%
  +{}{\PrintAuthors}{author}
  +{,}{ \textit}{title}
  +{}{ \parenthesize}{date}
  +{,}{ arXiv }{eprint}
}

\usepackage{tikz}

\usepackage{comment}

\usepackage{mathtools}

\usepackage{xcolor}
\definecolor{darkgreen}{rgb}{0,0.5,0}
\definecolor{darkblue}{rgb}{0,0.1,0.5}

\usepackage{hyperref}
\hypersetup{
    colorlinks=true, % make the links colored
    linkcolor=darkblue, % color TOC links in blue
    urlcolor=blue, % color URLs in red
    linktoc=all, % 'all' will create links for everything in the TOC
    citecolor=darkgreen
    }
    
\usepackage{enumerate}

\usepackage{tikz-cd}
\usepackage[matrix,arrow,curve,frame]{xy}

\usepackage{thmtools}
\newtheoremstyle{introTheorems}% name of the style to be used
  {\topsep}% measure of space to leave above the theorem. E.g.: 3pt
  {\topsep}% measure of space to leave below the theorem. E.g.: 3pt
  {\itshape}% name of font to use in the body of the theorem
  {0pt}% measure of space to indent
  {\bfseries}% name of head font
  {}
  { }% space after theorem head; " " = normal interword space
  {\thmname{#1}
  \textnormal{\bf\thmnote{#3}.\!\!}
  }

\newtheorem{theorem}{Theorem}[section]
\newtheorem{assumption}[theorem]{Assumption}

\newtheorem{conjecture}[theorem]{Conjecture}
\newtheorem{corollary}[theorem]{Corollary}
\newtheorem{definition}[theorem]{Definition}
\newtheorem{example}[theorem]{Example}
\newtheorem{lemma}[theorem]{Lemma}
\newtheorem{problem}[theorem]{Problem}

\newtheorem{remark}[theorem]{Remark}
\theoremstyle{introTheorems}
\newtheorem{introTheorem}{Theorem}
\newtheorem{introLemma}{Lemma}

\newcommand{\Z}{\mathbb{Z}}
\newcommand{\R}{\mathbb{R}}
\newcommand{\C}{\mathbb{C}}

\newcommand{\CC}{\mathcal{C}}
\newcommand{\BB}{\mathcal{B}}
\newcommand{\UU}{\mathcal{D}}
\newcommand{\V}{\mathcal{V}}
\newcommand{\W}{\mathcal{W}}
\newcommand{\cZ}{\mathcal{Z}}
\newcommand{\Schauenburg}{\mathcal{S}}

\newcommand{\Vect}{\mathrm{Vect}}
\newcommand{\Res}{\mathrm{Res}}
\newcommand{\Mod}{\mathrm{Mod}}
\newcommand{\Comod}{\mathrm{Comod}}
\newcommand{\RepVOA}{\mathrm{Rep}}
\newcommand{\ord}{\mathrm{ord}}
\newcommand{\id}{\mathrm{id}}
\newcommand{\Ind}{\mathrm{ind}}
\newcommand{\im}{Im}
\renewcommand{\i}{\mathrm{i}}
\renewcommand{\1}{\mathrm{1}}
\renewcommand{\d}{\mathrm{d}}
\newcommand{\loc}{\mathrm{loc}}
\renewcommand{\split}{\mathrm{split}}

\newcommand{\prim}{\mathrm{prim}}
\newcommand{\Ext}{\mathrm{Ext}}
\newcommand{\Hom}{\mathrm{Hom}} 
\newcommand{\forget}{\mathrm{forget}}
\newcommand{\FPdim}{\mathrm{FPdim}}

\renewcommand{\sl}{\mathfrak{sl}}

\newcommand{\zem}{\mathfrak{z}}
\renewcommand{\l}{\mathfrak{l}}
\newcommand{\g}{\mathfrak{g}}
\newcommand{\Nichols}{\mathfrak{N}}
\newcommand{\NicholsOf}{\mathfrak{B}}

\newcommand{\YD}[1]{{}_{#1}^{#1}\mathcal{YD}}

\renewcommand{\Schauenburg}{\mathcal{S}}

\newcommand{\Heis}{\pi}

\newcommand{\catC}{\mathcal{C}}
\newcommand{\catD}{\mathcal{D}}

\newcommand{\rev}{\mathrm{rev}}

\newcommand{\NicholsScreenings}{\mathfrak{Z}}

\renewcommand{\catC}{\mathcal{C}}
\renewcommand{\catD}{\mathcal{D}}

\title{A conditional algebraic proof of the logarithmic Kazhdan-Lusztig correspondence}
\author{Simon D. Lentner, University of Hamburg 
}
\date{}

\begin{document}

\maketitle

\begin{abstract}
The logarithmic Kazhdan-Lusztig correspondence is a conjectural equivalence between braided tensor categories of representations of small quantum groups and representations of certain vertex operator algebras.

    In this article we prove such an equivalence, and more general versions, using mainly algebraic arguments that characterize the representation category of the quantum group by quantities that are accessible on the vertex algebra side. Our proof is conditional on suitable analytic properties of the vertex algebra and its representation category. More precisely, we assume that it is a finite braided rigid monoidal category where the Frobenius-Perron dimensions are given by asymptotics of analytic characters. 
\end{abstract}

\tableofcontents

\section{Introduction}

\subsection{Conventions}

We are working over the field of complex numbers $\C$. All our categories are $\C$-linear and abelian. We use the term \emph{tensor category} for finite rigid monoidal categories as in \cite{EGNO15}. We distinguish the three types of representations: For modules over a commutative algebra $A$ in a braided monoidal category~$\UU$ with the tensor product $\otimes_A$ we use the notation $\UU_A$. For modules over a bialgebra $B$ in a braided monoidal category $\CC$ with the tensor product $\otimes$ in $\CC$ we use the notation $\Mod(B)(\CC)$. As linear categories, these notations coincide. For representations of a vertex operator algebra $\V$ we use the notation $\RepVOA(\V)$ with tensor product $\boxtimes_\V$.

\subsection{Background}

A \emph{vertex operator algebra} $\V$ is, very roughly speaking, a commutative algebra whose multiplication depends analytically on a complex variable $z$. It is a mathematical axiomatization of the chiral parts of a $2$-dimensional conformal quantum field theory. Its origin in mathematics is due to its distinguished role in the celebrated proof of the \emph{moonshine conjecture}, a surprising numerical coincidence between coefficients of the modular $j$-function and the dimensions of simple modules of the sporadic monster group. Mathematical textbooks on vertex algebras are  \cites{Kac97,FBZ04}, a broad introduction to moonshine is~\cite{Gan06}.

There is a notion of \emph{vertex modules}, which features under suitable assumptions a tensor product with a commutativity law, so to a well-behaved vertex operator algebra $\V$ there is associated a \emph{braided tensor category} $\RepVOA(\V)$, see \cites{HL94,HLZ06}. The data visible on the category level still retains some of the analytic aspects of the vertex operator algebra, for example, the braiding is related to the multivaluedness of certain analytic functions and the $\mathrm{SL}_2(\Z)$-action on certain $\Hom$-spaces captures the transformation behavior of certain modular forms, called $q$-characters, attached to each vertex module. These are particular instances of a topological field theory picture, which becomes more involved in the non-semisimple case.

\medskip

On the algebra side, a very interesting class of braided tensor categories is given by representations of the \emph{quantum group} $U_q(\g)$, which is the universal enveloping algebra of a semisimple complex finite-dimensional Lie algebra $\g$, deformed by a generic complex parameter $q$. Its category of representations is a semisimple infinite braided tensor category, which is as an abelian category the category of representations of $\g$, but the braiding and associator are nontrivial expressions involving $q$. They are related to the monodromy of the Knishnik-Zomolochikov partial differential equation, which was the original motivation for defining quantum groups in \cite{Drin89}. For an introduction to these topics and its application to produce topological invariants such as the Jones polynomial, the reader is referred to the textbook \cite{Kas97}.

The \emph{Kazhdan-Lusztig correspondence} \cite{KL93} is an equivalence of braided tensor categories of representations of $U_q(\g)$ and of a vertex operator algebra associated to the affine Lie algebra $\hat{\g}$ at a level $\kappa$ related to $q$, at generic values. In some sense, this gives a physical context to the observations above.

\medskip

If $q$ is specified to an $\ell$-th root of unity, then the quantum group $U_q(\g)$ admits a finite-dimensional quotient $u_q(\g)$, called small quantum group \cite{Lusz89}. Its category of representations is a nonsemisimple finite braided tensor category, related to the representations of $\g$ in finite characteristic.   

The \emph{logarithmic Kazhdan-Lusztig correspondence}$\,$ \cites{FGST05, FT10, AM14, Len21, Sug21, CLR23} gives a conjectural construction of a vertex operator algebra $\W$, whose braided tensor category of representations is equivalent to the category of representations of the small quantum group $u_q(\g)$, or of generalized versions thereof. For more details, contributions and known cases, the reader is referred e.g. to our recent article \cite{CLR23} and references therein. 

\subsection{Content of this article}

The present article builds on the approach we developed in \cite{CLR23} in order to prove the logarithmic Kazhdan-Lusztig correspondence in small cases, such as~$\sl_2$. The idea is to classify braided tensor categories $\UU$ that contain a commutative algebra $A$ with a known category of local $A$-modules, which is the a-priori knowledge on the vertex algebra side in our situation. In \cite{CLR23} we had assumed in addition that $\UU$ is known as abelian category (from earlier by-hand computations) and then we reconstructed the braided tensor structure. In the present article, we will reconstruct $\UU$ merely from information on certain $\Ext^1$-groups, using a Tannaka-Krein type reconstruction result we recently proved for this purpose \cite{LM24} and deep known results about Nichols algebras. These $\Ext^1$-groups in turn can be obtained systematically on the vertex algebra side.

Our main application is the computation of the category of modules of a vertex algebra algebra $\W$ defined as a kernel of screening operators inside a lattice vertex algebra $\V_\Lambda$. We prove that the category of representation of a certain Nichols algebra of diagonal type, which we have previously shown to be the algebra of screening operators \cites{ST12,Len21}, is contained in the category of twisted $\V$-modules over $\W$. From this it follows from the Schauenburg functor in \cite{CLR23} that the category of representations of a generalized small quantum group \cites{Heck10, AY15, CLR23} is a subquotient of the category of $\W$-modules. This result also contains many interesting cases beyond finite-dimensional semisimple Lie algebra and Lie superalgebras, and it generalizes beyond kernels of screenings inside lattice vertex algebras.

If now in addition, for example, the Frobenius-Perron dimensions on both sides of the logarithmic Kazhdan-Lusztig correspondence are known to match, then we have an equivalence of braided tensor categories. 
In particular, we can prove the logarithmic Kazhdan Lusztig correspondence in all cases related to simply-laced Lie algebras \cites{FGST05, FT10,AM14,Len21, Sug21}, if we allow ourselves to assume the \emph{existence} of a tensor product and a braiding on the (analytic) vertex algebra side in accordance with \cite{HLZ06}, that enjoys favorable properties, namely finiteness, rigidity, and the quantum dimensions computed in \cite{BM17} coinciding with categorical dimensions (nondegeneracy follows in our proof). Proving these existence results on the vertex algebra side are difficult question of a very different nature, with major recent developments in \cites{CY21, CJORY21, McR21, CMY21, CMY22, CMY23, NORW24}.\footnote{The experts will agree that rigidity and even more quantum dimensions, which in the semisimple case depends on the Verlinde formula, are quite substantial assumptions. For the role of rigidity in this article and possible resolutions see Remark \ref{problem_rigidity}.}
Note that our approach also considerably simplifies proofs in existing cases, and could be applied to other cases where the vertex tensor category is already established, see the table in \cite{CMSY24}.\\

We now discuss the content of the article in more detail: \\

In Section \ref{sec_comm} we recall our main setup from \cite{CLR23}: Let $\UU$ be a braided monoidal category. Assume that $A$ is a commutative algebra in $\UU$, then there is an associated monoidal category of modules $\UU_A$ and a braided monoidal category of local modules $\UU_A^\loc$. There is a tensor functor, given by induction $A\otimes(-)$, and its right adjoint, the lax tensor functor forgetting the $A$-action.
$$
%\hspace{0.8cm}
 \begin{tikzcd}[row sep=10ex, column sep=15ex]
   {\UU} \arrow[shift left=1]{r}{A\otimes(-)}  
   & 
   \arrow[shift left=1]{l}{\mathrm{forget}_A} 
   \BB  
   %\arrow[equal]{r} &     
   =\UU_A \\
   &\arrow[hookrightarrow, shift left=6]{u}{}
   %\arrow[below, dotted, shift left=2]{ul}{} %{\coVerma}
   \CC 
   %\arrow[equal]{r}& 
   =\UU_A^\loc
\end{tikzcd}
$$
From the perspective in this article, $\UU$ is difficult, unknown and non-semisimple and the goal is to reconstructed it from $\CC$, which is known and semisimple, and from additional knowledge about $\BB$, in our case the group $\Ext^1$ of indecomposable extension in $\BB$ of simple objects contained in the subcategory $\CC$. 

Some particular technical preliminaries we will require are as follows: In Section \ref{sec_commAlgebras} we recall from \cites{EO21,CMSY24,CSZ25} under which circumstances $\catC_A$ is rigid if $\catC$ is. In Section \ref{sec_simpleCurrent} we recall for simple current extensions that induction is exact, and that simples and indecomposables extensions of simples are preserved in fixpoint free cases. This is used later-on to be flexible in our choice of Cartan part. In Section~\ref{sec_FP} we recall from \cites{EGNO15,LW21, SY24} properties of the Frobenius-Perron dimension, in particular a formula for $\catC_A$ and bounds for $\catC_A^\loc$, even if $\catC$ has a-priori a degenerate braiding. \\

%In Section \ref{sec_GV} we briefly introduce Grothendieck-Verdier (GV) duality \cite{BD13} for monoidal categories. This is a weaker duality notion than rigidity, in that the dual of the tensor unit need not be the tensor unit (rather the so-called dualizing object) and more severely the dual is not compatible with the tensor product. This weaker duality is present a-priori for every vertex algebra \cite{ALSW21}. Its use in the present article is for one that it makes the tensor product right-exact, and that it allows to express tensor functors in terms of adjoint algebras, in the theory of Hopf monads \cite{BLV11}, see below. In  Section \ref{sec_GVLemma} we briefly study how rigid objects inside a GV category behave, in particular we show that for a rigid category $\mathcal{X}$ inside a GV category $\catC$ the category $\cZ(\mathcal{X},\catC)$ of objects in $\mathcal{X}$ with a halfbraiding in $\catC$ is rigid, and assuming pivotality the category $\cZ(\catC,\mathcal{X})$ of objects in $\catC$ with a halfbraiding in $\mathcal{X}$ is Grothendieck-Verdier.

In Section \ref{sec_reconstruction} we first review one of our key ideas in \cite{CLR23}: A general braided monoidal category $\UU$ with a commutative algebra $A$ can be reconstructed as a relative Drinfeld center $\cZ_\CC(\BB)$ \cite{LW22} of $\BB=\UU_A$ relative to $\CC=\UU_A^\loc$, using a version of a functor introduced by Schauenburg \cite{Sch01}. \bigskip

The two fundamental example classes we have in mind for such a situation~$\UU,\UU_A,\UU_A^\loc$ are the following. Proving the logarithmic Kazhdan-Lusztig correspondence means to bring these two into match:

\begin{itemize}
\item For a (quasi-)Hopf algebra with $R$-matrix, the category of representations is a braided monoidal category. If $u_q(\g)$ is the small quantum group associated to a semisimple finite-dimensional complex Lie algebra $\g$ and an $\ell$-th root of unity, then the dual Verma module of weight $0$ is a commutative algebra $A$, the modules over $A$ are equivalent (by induction) to the category of representations of the Borel part $u_q(\g)^{\geq 0}$, and the local modules are equivalent to the category of representations of the Cartan part $u_q(\g)^0$, which is the category $\Vect_\Gamma$ of vector spaces graded by weights or cosets of weights. Constructing the category of representations of $u_q(\g)$ as a relative Drinfeld center \cites{Lau20,LW22} is a categorical version of the construction of $u_q(\g)$ as a Drinfeld double of $u_q(\g)^{\geq 0}$ while merging the two copies of $u_q(\g)^0$.

The Borel part, in turn, has the form of a semidirect product or Radford biproduct of a  \emph{Nichols algebra} over the Cartan part. The Nichols algebra~$\NicholsOf(X)$ is a certain universal graded Hopf algebra attached to an object $X$ in a braided monoidal category $\CC$, such that $X$ becomes the  homogeneous component of $\NicholsOf(X)$ of degree $1$ and the space of primitive elements. A main textbook on this topic is \cite{HS20}. Reconstructing an arbitrary pointed Hopf algebra from the corresponding Nichols algebra in $\Vect_\Gamma$ is a main idea in the Andruskiewitsch-Schneider program \cite{AS10}. In a certain sense, one could describe our attempts to reconstruct $\UU$ as a categorical version and as a braided version (whence a Drinfeld center) of this program. 

If $\ell=2p$ has even order, then having a nondegenerate braiding requires that $\Vect_\Gamma$ has a nontrivial associator, which means that $u_q(\g)^{0}$ is a quasi-Hopf algebra. Then the relative center of the Nichols algebra representation category can be read as a clean categorical definition of a factorizable quasi-Hopf algebra $\tilde{u}_q(\g)$ in the even order case.\footnote{We expect this construction to coincide with the earlier alternatives constructions of this quasi-Hopf algebra in \cites{CGR20, GLO18, Neg21}, but the details would need to be worked out. }

\item For a $C_2$-cofinite vertex operator algebra $\mathcal{V}$, the category of representations is a braided monoidal category \cite{HLZ06}, with tensor product $\boxtimes_\mathcal{V}$ and braiding defined by analytic means. If $\mathcal{W}\subset \mathcal{V}$ is a conformal embedding of  vertex operator algebras, then $A=\mathcal{V}$ can be regarded as a commutative algebra in the category $\UU$ of representations of $\mathcal{W}$, and the local modules over $A$ recover the category $\CC$ of representations of $\mathcal{V}$.

The category $\BB$ of all $A$-modules in $\CC$ is less studied on the vertex algebra side, it consists of $\V$-modules which are nonlocal (that is, the vertex operator contains multivalued functions) but become local when restricted to $\W$. For an orbifold by a finite group $\W=\V^G$, such modules are called \emph{twisted modules} and it can be shown that they are classified in terms of \emph{$g$-twisted modules} with a precisely prescribed monodromy depending on a group element $g$ \cites{FLM88, DL96, DLM96}. 

The vertex algebras $\mathcal{W}$ relevant to the logarithmic Kazhdan-Lusztig conjecture are defined as vertex subalgebras of easier vertex algebras $\V$, this is called free field realization. In our case, $\V=\mathcal{V}_\Lambda$ is a lattice vertex algebra and its category of representations is $\CC=\Vect_\Gamma$ associated to the discriminant form of the lattice $\Gamma=\Lambda^*/\Lambda$. The vertex-subalgebra $\mathcal{W}$ is defined as a kernel of nonlocal screening operators $\zem_1,\ldots,\zem_n$ associated to elements $\alpha_1,\ldots,\alpha_n\in \Lambda^*$, which we have shown in \cite{Len21} to fulfill the relations of the generators in the Nichols algebra $\NicholsScreenings$ with respect to the braiding $\Vect_\Gamma$.\footnote{Let us remark that we will not use this piece of information in the present article, rather, we use directly the universal property of the Nichols algebra and algebraic arguments.}  The category $\BB$ consists of indecomposable extensions of local $\V_\Lambda$-modules. We would propose in general to interpret $\W$ as a kind of orbifold of $\V$ by the Nichols algebra of screenings $\NicholsScreenings$. Ultimately, we expect the following big picture statement to be true. Its interpretation would be that all twisted modules are ''$M$-twisted'' for some representation $M$ of the Nichols algebra of screenings, and all such modules can be obtained by a deformation procedure (because the action is inner).
\end{itemize}

\begin{problem}\label{prob_bigPicture}
Let $\W\subset \V$ be a kernel of a set of screening operators $\zem_1,\ldots,\zem_n$. Prove that the tensor category $\BB$ of modules over the commutative algebra $A=\V$ inside the category of $\W$-modules (''twisted $\V$-modules over $\W$'') is equivalent to the tensor category of modules over the Nichols algebra $\NicholsScreenings$ generated by the screening operators inside $\RepVOA(\V)$.

Note that there is indeed a good candidate for a fully faithful exact functor 
$$F:\Mod(\NicholsScreenings)(\RepVOA(\V))\to \RepVOA(\W)_A$$ 
as follows: Let again $\V$ be a lattice vertex algebra. Let $\mathcal{M}$ be a $\V$-module with an additional action 
$$\mathcal{X}\boxtimes_\V \mathcal{M}\to \mathcal{M}$$
of the Nichols algebra $\NicholsOf(\mathcal{X})$ of the $\V$-module $\mathcal{X}=\bigoplus_i \V_{\alpha_i+\Lambda}$. Then we can consider the \emph{$\Delta$-deformation} $\widetilde{\mathcal{M}}$ of $\mathcal{M}$ in the sense of \cites{DLM96, Li97, FFHST02, AM09}. 
%by modifying the action to
%$$\Y_{\widetilde{\mathcal{M}}}(a,z)
%:=\Y_{\cM}(\Delta(\zem,z)a,z),\qquad
%\Delta(\zem,z)=\exp\left(\zem_0\log(z)+\sum_{n=1}^\infty \frac{\zem_n}%{-n}(-z)^{-n}\right)
%$$
%according to morphism $\mathcal{X}\boxtimes_\V \mathcal{M}\to \mathcal{M}$ given by the action. 
The modified action of $\V$ on $\widetilde{\mathcal{M}}$ contains logarithms coupled to the screening operators $\zem_i$ associated to $\alpha_i$, so this should produce a $\V$-twisted module which is single-valued over the kernel $\W$ of these screening operators. The question is whether this functor is an equivalence of categories.
\end{problem}

Hence the goal of our article is to first reconstruct $\BB$ as representations of some Hopf algebra $\Nichols$, which we then prove to coincide with the algebra of screenings $\NicholsScreenings$ by algebraic methods. We then use the previous results on the Schauenburg functor to prove that $\UU$ is its relative Drinfeld center. We now discuss in more detail how this is achieved:\\

In Section \ref{sec_splitting}, for a monoidal category $\BB$ and a central full subcategory $\CC$, we define the subcategory $\BB^{\split}$ consisting of all objects in $\BB$ admitting a composition series with objects in $\CC$. Not surprisingly, this is a tensor subcategory of $\BB$ and there is a tensor functor $\l:\BB^\split\to \BB$ sending an object to the direct sum of its composition factors, see Lemma \ref{lm_splittingSubcat}.
By our recent result \cite{LM24}, any exact faithful tensor functor admitting a tensor functor section leads to a reconstruction by a categorical Hopf algebra (and vice versa); this is a categorical version of Tannaka-Krein reconstruction and the Radford projection theorem. In particular, this can be applied to $\BB^\split$ and yields:
%Rigidity??
%TannakaKrein: Exact faithful from locally-finite to Vect
%OUR ARTICLE: Exact faithful module functor, 
%but then again the braiding....? 

%? Fibre functor is faithful exact. True by forgetting
%? Embedding is right exact. True by subcat
%? Finite of C suffices?
%? Rigidity necessary?? (compare recent article) (then Hopf not)
%? Corep instad of Rep
\begin{introLemma}[\ref{lm_splitLM}]
Let $\CC$ be a braided tensor category, which is a central full abelian subcategory of a tensor category $\BB$, then there exists a Hopf algebra $\Nichols$ in $\CC$ and an equivalence of monoidal categories
$$\BB^{\split}\cong\Comod(\Nichols)(\CC).$$  
Here, the functor $\l:\BB^{\split}\to \CC$ corresponds to the functor forgetting the coaction of $\Nichols$ and the full embedding $\CC\hookrightarrow \BB^{\split}$ corresponds to the functor endowing an object in $\CC$ with the trivial $\Nichols$-coaction via the unit $\eta_\Nichols:1\to \Nichols$.
\end{introLemma}

As we applied the reconstruction to $\BB^\split$, by construction $\Nichols$ is a \emph{connected coalgebra} in $\CC$, that is, all simple $\Nichols$-comodules in $\CC$ are trivial $\Nichols$-comodules.\footnote{In the dual situation, one could say $\Nichols$ is a \emph{local algebra} in the standard algebra sense \cite{Zimm14}, which is unrelated to the notion of local modules used above.}  

\begin{example}
In semisimple situations, for example for group orbifold vertex algebras, we have $\BB^{\split}=\CC$ and $\BB$ contains new simple objects. On the other hand in the logarithmic Kazhdan-Lusztig context we have or expect that $\BB^{\split}=\BB$, so all twisted modules are extensions of local modules.\footnote{Recently \cite{CN24} Section 5.5 give very interesting criteria when this situation arises.}

In general we have a combination of these two situations, which can be interlocked: Consider the Nichols algebra of the non-invertible $g$-twisted object in the Tambara-Yamagami category. On the vertex algebra side, this corresponds, supposedly, to a $\Z_2$-orbifold and then a kernel of a screening with respect to an intertwiner in the twisted sector. In this case, one would expect nontrivial extensions between local modules and this new $g$-twisted simple modules. 
\end{example}

In Section \ref{sec_NicholsExt} we relate further properties of the category $\BB^{\split}$ to properties of the coalgebra $\Nichols$. We discuss a categorical version of Hochschild cohomology related to $\Ext^1$ in $\Mod(\Nichols)$, in complete analogy to the classical versions \cite{Zimm14}, and a dual version related to $\Ext^1$ in $\Comod(\Nichols)$. In particular, it is not difficult to see that the extensions between trivial $\Nichols$-comodules are in bijection with primitive elements in $\Nichols$:

\begin{introLemma}[\ref{lm_CoExt}]
 The vector space $\Ext^{1}_{\Comod(\Nichols),\,\CC\,\split}(\1_\eta,M_\eta)$ of $\CC$-split extensions between trivial $\Nichols$-comodules $\1_\eta,M_\eta$ is isomorphic to the vector space $\Hom_\CC(M,\Nichols^{\prim})$, where  we define the $\CC$-subobject $\Nichols^{\prim}$ of $\Nichols$ as the equalizer of $\Delta_{\Nichols}$ and $\id\otimes \eta+\eta \otimes \id$.
\end{introLemma}

In Section \ref{sec_NicholsArguments} we study the Hopf algebra $\Nichols$ with methods from the theory of Nichols algebras. As a technical simplification, we restrict ourselves to the case of a  $\CC=\Vect_\Gamma$ and as in \cite{AG17} we extend the group $\tilde{\Gamma}\to \Gamma$ (at the cost of modularity) in order to have a trivializable associator $\omega=1$ and thus a fibre functor to $\Vect$. Hence, $\Nichols$ can be viewed as a algebra and coalgebra in the ordinary sense and previous work on Nichols algebras applies. In Lemma \ref{lm_fibreConnected} we check that the fibre functor sends a coalgebra $\Nichols$ in $\CC$ that is connected in the categorical sense (that is: all $\Nichols$-comodules in $\CC$ are simple $\CC$-objects with trivial coaction) to a coalgebra in $\Vect$ that is connected in the ordinary sense (that is: the coradical is trivial).\footnote{We remark that this is a trick to connect to existing results and avoid more systematic work: ultimately one should develop a framework of (co)radical filtration and local algebras resp. connected coalgebras in the categorical setting, see Problem \ref{problem_connected}}

At this point we know that $\Nichols$ is a connected finite-dimensional Hopf algebra in $\CC$ with a diagonal braiding. We invoke two major results in the context of the Andruskiewitsch-Schneider program, both proven in full generality for $\CC=\Vect_\Gamma$ by I. Angiono et.al. \cites{An13, AKM15}:

\begin{itemize}
\item $\Nichols$ is graded with respect to the coradical filtration. Note that this would also follow in sufficiently unrolled situations in the sense of \cite{CLR23} from the $\Gamma$-grading. Technically, it needs to be checked that the arguments do not depend on the extension $\tilde{\Gamma}$, but this has been done in \cite{AG17}. Counterexamples at this point would be the infinite-dimensional Nichols algebra $\C[X]/X^p \otimes \C[Y]$, which can be lifted to the Hopf algebra $\C[X]$ with an additional primitive $X^p=Y$, a so-called pre-Nichols algebra.
\item $\Nichols$ is generated in coradical degree $1$, that is, there are no higher generators, which would be undetectable in $\Ext^1$. Note that in finite dimension, this argument can be replaced by a dimension matching. Counterexamples at this point would be Lusztig's infinite dimensional quantum group containing divided powers as generators of higher degree, for example, dual to the situation in the previous bullet, a generator $Y=X^{(p)}$ with 
$$\Delta(Y)=1\otimes Y+\sum_{i=1}^{p-1} \frac{X^i}{[i]_q!}\otimes \frac{X^{p-i}}{[p-i]_q!}+Y\otimes 1.$$
Accordingly, this situation should appear e.g. for the category of modules over the Virasoro algebra.
\end{itemize}

The Nichols algebra $\mathfrak{B}(X)$ of an object $X$ in $\CC$ can be defined as the unique connected coradically-graded bialgebra, whose homogeneous component in degree $1$ is $X$. 
Combined with the previous results this gives the following remarkable reconstruction result:

\begin{introTheorem}[\ref{thm_ReconstructNichols2}]
        Let $\BB$ be a finite monoidal category and $\CC$ a central full tensor subcategory, which is of the form $\CC=\Vect_\Gamma^{Q}$. Then the tensor subcategory $\BB^{\split}$ is equivalent to $\Comod(\Nichols)(\CC)$ where $\Nichols=\NicholsOf(X)$ is the Nichols algebra  of the object 
        $$X=\bigoplus_{a\in\Gamma} \Ext^1_{\BB}(\C_a,\1) \C_a.$$
\end{introTheorem}

Of course, this role of the Nichols algebra is very similar to the role it plays in the classification of pointed Hopf algebras \cite{AS10}: By definition, the Hopf algebra contains a lifting of the Nichols algebra $\NicholsOf(X)$ of primitive elements. The two results quoted above show that the Nichols algebra itself has no nontrivial liftings and that no other generators are present. Hence the only thing left is to compute the nontrivial liftings between $\NicholsOf(X)$ and $\C[\Gamma]$, but in our context the splitting condition rules out such liftings.

The striking properties of Nichols algebras is that the finite-dimensional Nichols algebras can be classified in terms of root systems \cites{Heck06, Heck09, AHS10, CH15}. Hence there is an explicit list of possible scenarios, containing quantum groups, quantum supergroups and exceptional cases, see  \cites{Heck06, AA17, FL22}.

\bigskip

In Section \ref{sec_proof1} we turn the results so far into a  conditional proof of the Kazhdan Lusztig correspondence. The idea is to produce elements in $\Ext^1(\C_a,1)$ for twisted $\V$-modules over the kernel of screenings $\W$ from the knowledge of $\Ext^1(\C_a,1)$ for twisted $\V$-modules over the kernel of a  single screening. This is a more restrictive condition on a module, and hence gives a full tensor subcategory, see Lemma \ref{lm_produceExt}. In turn, these $\Ext^1(\C_a,1)$ over the kernel of a single screening can be deduced from the the known case $\sl_2$.

\begin{theorem}
Let $\Lambda$ be an even integral lattice and $\alpha_1,\cdots,\alpha_n\in\Lambda$ elements, such that the cosets in $\Lambda^*/\Lambda$ are pairwise distinct and nonzero, and such that the inner product $(\alpha_i,\alpha_i)=2/p_i$ for integers $p_i>1$. Let $\mathcal{V}_{\Lambda}$ be the lattice vertex algebra, with a conformal structure such that the conformal dimension $h(\alpha_i)=1$ and let $\mathcal{W}$ the kernel of the screening operators $\zem_1,\ldots,\zem_n$ associated to $\alpha_1,\cdots,\alpha_n$. Assume that the category of representations of $\mathcal{W}$ becomes a braided monoidal category with respect to the HLZ-construction. Then the category $\BB=\RepVOA_\mathcal{V}(\mathcal{W})$ of twisted $\mathcal{V}$-modules over $\mathcal{W}$ contains as a full subcategory the representations of the Nichols algebra $\NicholsOf(q)$ associated to the diagonal braiding $q_{ij}=e^{\pi\i (\alpha_i,\alpha_j)}$ in $\Vect_{\Lambda^*/\Lambda}$.
\end{theorem}

We could now attempt to conclude equivalence of these categories by comparing dimensions. Combined with the Schauenburg functor, the overall result is then as follows:   

\begin{introTheorem}[\ref{thm_KLviaFP}]
Let $\W$ by a kernel of screening operators in a lattice vertex algebra $\V_\Lambda$. Assume that $\RepVOA(\W)$ is a braided tensor category. Assume the Frobenius-Perron dimension of $A=\V_\Lambda$ over $\W$ is equal to the dimension of the diagonal Nichols algebra $\NicholsOf(q)$. Then there is an equivalence of tensor categories resp. braided tensor categories 
\begin{align*}
\RepVOA_\W(\V)&= \Mod(\NicholsOf(q))(\CC) \\
\RepVOA(\W)&=\cZ_\CC(\Mod(\NicholsOf(q))(\CC)).
\end{align*}
The last category should be considered a generalized quantum group.
\end{introTheorem}

In Section \ref{sec_CharacterFormula} we demonstrate one of the ideas listed in Remark \ref{rem_FPproof_whatnow} to make this result effective, namely by an analytic computation of the Frobenius-Perron dimension:  For the Feigin-Tipunin algebra $\mathcal{W}_p(\g)$ with $\g$ simply-laced, the $q$-characters and their asymptotics are known \cites{FT10, BM17, Sug21} and they indeed match the dimension of the respective Nichols algebra. Hence we get the following conditional proof of the logarithmic Kazhdan-Lusztig conjecture:

\begin{corollary}
Assume that the Feigin-Tipunin algebra $\mathcal{W}_p(\g)$ is $C_2$-cofinite and the quantum dimension of $A$ (in the vertex algebra sense) coincides with the Frobenius Perron dimension of $A$ (in the categorical sense). Then the category of representations of $\mathcal{W}_p(\g)$ is equivalent, as braided monoidal category, to the category of representation of the quasi-Hopf version of the small quantum group $\tilde{u}_q(\g)$ with $q=e^{\frac{2\pi\i}{2p}}$ and coradical $\Lambda^*/\Lambda$, see Section \ref{sec_generalizedQuantumGroup}.
\end{corollary}

\begin{problem}\label{problem_rigidity}
We have assumed rigidity of $\RepVOA(\W)$ and we discuss its role and how it could be avoided. 

\begin{itemize}
    \item The reconstruction of the Nichols algebra \cite{LM24} in Section \ref{sec_splitting} and the Nichols algebra arguments in Section \ref{sec_NicholsArguments} do not essentially use rigidity, in particular all the arguments can be done for bialgebras instead of Hopf algebras, and moreover by \cite{Mon93} any connected bialgebra is in fact a Hopf algebra (this is explicitly clear for the Nichols algebra). Note that this matches the nice and  more general result in \cite{CMSY24}~Theorem~3.14.
    \item The final argument uses Frobenius-Perron dimensions, which depends in its present form on rigidity. It seems reasonable that there is a submultiplicative Frobenius-Perron dimension for Grothendieck-Verdier categories \cites{BD13, FSSW24}, having merely a right exact tensor product, which is the weaker duality we always have in vertex operator algebras by the contragradient module \cite{ALSW21}. Note that the technical notion of transitivity in \cite{EGNO15} Section 4.5 can be proven in such a setting. A good candidate would be the $C_1$-dimension, see Remark \ref{rem_FPproof_whatnow}.
    \item The Schauenburg functor being an equivalence in Theorem \ref{thm_Schauenburg} requires rigidity, while in \cite{CMSY24} this can be concluded if $\BB=\BB^\split$.
    %\item Note also the very recent result in 
\end{itemize}
Finiteness is used for Frobenius-Perron dimensions, and it is also used for the existence of a right adjoint (for a right-exact functor) and of a coend in \cite{LM24}.

\end{problem}

Acknowledgement: Many thanks to Thomas Creutzig for valuable insights and remarks and for hospitality at University of Alberta and Erlangen.  
%and for providing the perfect private surrounding to write this article. 
The article was finished during a stay at RIMS Kyoto, thanks to T. Arakawa for hospitality and helpful remarks. Many thanks for remarks also to  Jinwei Yang, Robert McRae, Shegonori Nakatsuka, Shoma Sugimoto and Christoph Schweigert.

%%%%%%%%%%%%%%%%%%%%%%%%%%%%%%%%%%%%%%%%%%%%%%%%%%%%%%%%

\newcommand{\one}{1}
\newcommand{\cat}{\mathcal{D}}
\newcommand{\catA}{\mathcal{D}_A}
\newcommand{\catAloc}{\mathcal{D}_A^\loc}

\section{Categorical preliminaries}\label{sec_comm}

\subsection{Braided monoidal categories}

Let $\cat$ be a monoidal category. We denote the associativity isomorphism by $a_{X,Y,Z}:(X\otimes Y)\otimes Z\to X\otimes (Y\otimes Z)$ and we omit it from formulas. If $\cat$ is a braided monoidal category, we denote the braiding isomorphism by $c_{X,Y}:X\otimes Y\to Y\otimes X$. We reserve the word \emph{tensor category} for finite and rigid monoidal categories in the sense of \cite{EGNO15}.

As a first example, the following semisimple braided monoidal category is relevant to the present article:

\begin{example}[Quadratic Space]\label{ex_quadraticSpace}
Let $(\Gamma,+)$ be an abelian group. The $\C$-linear category $\Vect_\Gamma$ of $\Gamma$-graded vector spaces is semisimple and its simple objects are $1$-dimensional vector spaces $\C_a$ for any $a\in\Gamma$. This category becomes a braided monoidal category when endowed with the following coherence data

$$
\begin{tikzcd}
(\C_a\otimes \C_b)\otimes \C_c
\arrow[rr, dashed,"a_{\C_a,\C_b,\C_c}"]
\arrow[d,equal]
&&
\C_a\otimes (\C_b\otimes \C_c)
\arrow[d,equal]
\\
\C_{a+b+c}
\arrow{rr}{\omega(a,b,c)}
&&
\C_{a+b+c}
\end{tikzcd}
\begin{tikzcd}
\C_a\otimes \C_b
\arrow[rr, dashed,"c_{\C_a,\C_b}"]
\arrow[d,equal]
&&
\C_b\otimes \C_a
\arrow[d,equal]
\\
\C_{a+b}
\arrow{rr}{\sigma(a,b)}
&&
\C_{b+a}
\end{tikzcd}
$$
for  scalar function $\sigma:\Gamma\times\Gamma\to \C^\times$ and  $\omega:\Gamma\times\Gamma\times\Gamma\to \C^\times$ that fulfill  hexagon and pentagon relation. The trivial endofunctor with a nontrivial monoidal structure gives a nontrivial equivalence between the braided monoidal categories attached to different pairs $(\sigma,\omega)$.

There is a cohomology theory called \emph{abelian cohomology} \cite{MacL52}, where the compatible pairs $(\sigma,\omega)$ are  abelian $3$-cocycle, and the equivalence classes of pairs in the sense above are  abelian $3$-cohomology classes. It is known \cites{MacL52,JS93} that abelian $3$-cohomology classes are classified by quadratic forms $Q:\Gamma\to \C^ \times$, and we denote the associated bimultiplicative form by $B(a,b)=Q(a+b)Q(a)^{-1}Q(b)^{-1}$. An abelian $3$-cocycle $(\sigma,\omega)$ is related to this data by  $\sigma(a,a)=Q(a)$ and $\sigma(a,b)\sigma(b,a)=B(a,b)$ and $\omega$ can be thought to express to which extend the chosen $\sigma(a,b)$ is not a bimultiplicative form.

As a particularly relevant (infinite) example, let $\Gamma=\R^n$ with a choice of inner product and define $Q(v)=e^{\pi\i(v,v)}$, then $B(v,w)=e^{2\pi\i(v,w)}$. A choice of abelian $3$-cocycle is $\sigma(v,w)=e^{\pi\i(v,w)}$ and $\omega(u,v,w)=1$. In general, it is not always possible to choose $\omega=1$, if $\Gamma$ has $2$-torsion. 
\end{example}

\subsection{Commutative algebras}\label{sec_commAlgebras}

For the following definitions and theorems we refer the reader for example to \cite{FFR06}, to the textbook \cite{EGNO15}, to \cite{CKM24} and to our article \cite{CLR23} Section 2 and references therein.

\begin{definition}\label{def:alg}\textup{\cite{EGNO15} Definition 7.8.1}
An \emph{algebra} in a monoidal category~$\cat$ is an object $A$ in $\cat$ together with a morphism called multiplication
$\mu: A \otimes A \rightarrow A$
 and a morphism called unit
$\eta : \one \rightarrow A$,
satisfying associativity and unitality.

A \emph{commutative algebra} $A$ in a braided monoidal category $\cat$ is an algebra in~$\cat$ such that the following morphisms coincide 
$$\big(A\otimes A \stackrel{\mu}{\longrightarrow} A\big)
\;=\;
\big(A\otimes A\stackrel{c_{A,A}}{\longrightarrow} A\otimes A\stackrel{\mu}{\longrightarrow} A\big)
$$
\end{definition}

A \emph{module} $(X,\rho_X)$ over an algebra $A$ in a monoidal category  is an object $X\in\cat$ and a morphism called action $\rho_X:A\otimes X\to X$ in $\cat$, such that again the usual compatibility with  $\mu$ holds.
%\begin{equation}
%\xymatrix{
%A \otimes (A\otimes X)    \ar[rr]^{  a_{A, A, X} } %\ar[d]_{\id_A \otimes m_X } 
%&&  (A \otimes A) \otimes X  \ar[d]^{m \otimes \id_X}
%\\
%A\otimes X \ar[rd]_{m_X} && A\otimes X \ar[ld]^{m_X} \\
%& X &
%\\
%}
%\end{equation} and the following composition is the identity on %$X$
%\[
%X \xrightarrow{\ell_X} \one \otimes X \xrightarrow{u \otimes %\id_X} A \otimes X \xrightarrow{m_X} X.
%\] 
A morphism of $A$-modules $f:(X,\rho_X)\to (Y,\rho_Y)$ is a morphism in $\cat$, which is compatible with the action.
%\begin{equation}
%\begin{split}
%\xymatrix{
%A \otimes  X    \ar[rr]^{  \id_A \otimes f   } \ar[d]_{ m_X } 
%&&  A \otimes Y  \ar[d]^{m_Y }
%\\
%X \ar[rr]^{f} && Y 
%} 
%\end{split}
%\end{equation}

\begin{definition}
For $A$ an algebra in a monoidal category $\cat$, we denote by $\catA$ the category of $A$-modules in $\cat$ and by ${_A}\cat_A$ the category of $A$-bimodules.
\end{definition}

For an algebra $A$ in a monoidal category $\cat$, the category of $A$-bimodules is a monoidal category with tensor product $\otimes_A$, defined as a coequalizer as usual.  
For a commutative algebra $A$ in a braided monoidal category, an $A$-module can be turned into an $A$-bimodule, but there are choices of using over- or underbraiding. As a consequence, the situation for a general braiding is more complicated then in the classical case, where the braiding is symmetric:

\begin{definition}
For $A$ a commutative algebra in a braided monoidal category~$\cat$, we denote by~$\catAloc$ the subcategory of \emph{local} $A$-modules, which by definition means that the following morphisms coincide :
\begin{equation*}
\xymatrix{
A \otimes  X    \ar[rd]_{  \rho_X }  \ar[rr]^{ c_{X, A} \circ  c_{A, X}} && A\otimes X \ar[ld]^{\rho_X} \\
& X&  \\
} 
\end{equation*}
\end{definition}
\begin{theorem}[see e.g. \cite{CKM24} Thm. 2.53 and 2.55]
For a commutative algebra $A$ in a braided monoidal category,
the category of $A$-modules $\catA$ is a monoidal category with the tensor product $\otimes_A$. The subcategory of local modules $\catAloc$ is a braided monoidal category.
\end{theorem}

We have an adjoint pair of functors: Left-adjoint the right-exact induction functor $A\otimes(-):\,\UU\to \UU_A$, which is a tensor functor, and right-adjoint the left exact functor forgetting the $A$-action $\forget_A:\,\UU_A\to \UU$, which is a lax tensor functor. %All claims clear
We summarize all structures in the following diagram:
$$
%\hspace{0.8cm}
 \begin{tikzcd}[row sep=10ex, column sep=15ex]
   {\UU} \arrow[shift left=1]{r}{A\otimes(-)}  
   & 
   \arrow[shift left=1]{l}{\mathrm{forget}_A} 
   \BB  
   %\arrow[equal]{r} &     
   =\UU_A \\
   &\arrow[hookrightarrow, shift left=6]{u}{}
   %\arrow[below, dotted, shift left=2]{ul}{} %{\coVerma}
   \CC 
   %\arrow[equal]{r}& 
   =\UU_A^\loc
\end{tikzcd}
$$

We need results that ensure rigidity and finiteness of $\UU_A$:

\begin{theorem}[\cite{SY24} Theorem 5.5 and Corollary 5.14b]\label{thm_UATensorCat}
Let $A$ be an  commutative algebra in a braided tensor category $\UU$ (that is: finite and rigid). Assume $A$ is haploid, meaning  $\dim\mathrm{Hom}_\UU(1,A)=1$,  and exact, meaning $\UU_A$ is exact as a module category \cite{EO04}. Then $\UU_A$ is a tensor category and $\UU_A^\loc$ is a braided tensor category.
\end{theorem}

In \cite{EO21} and \cite{CMSY24} Section 2 several useful conditions for an algebra to be exact are given, for example if $A$ is self-dual (Frobenius) and generalizations thereof, which would in principle be sufficient for our purposes. During the work on this article, the following conjecture in \cite{EO21} was proven in \cite{CSZ25}, using a categorical version of the Jacobson radical, which should also be useful below.

\begin{theorem}[\cite{CSZ25} Theorem 7.1]\label{thm_UAexact}
If $A$ is a direct product of simple algebras, then $A$ is exact.  
\end{theorem}

The converse was already known to hold. In our application, these results are used as follows:
Suppose the unit in $\UU_A^\loc$ is known to be simple, then the algebra $A$ is simple as an algebra, and it is then also haploid because $\Hom_\UU(1,A)=\Hom_{\UU_A}(A,A)=\C$ as in the proof of \cite{CMSY24} Theorem 2.24. Then the two previous theorems combined give:

\begin{corollary}\label{cor_isTensorCat}
Suppose $\UU$ is a braided tensor category and  the unit in $\UU_A^\loc$ is simple, then $\UU_A$ is a tensor category and $\UU_A^\loc$ is a braided tensor category.
\end{corollary}

\begin{comment}
Note that the condition of self-duality can be replaced by the weaker condition that there is an embedding $A^*\hookrightarrow A\otimes X$ in  $\UU_A$  for some $X\in \UU$, see \cite{CMSY24} Theorem 2.21.

In particular the following version is suited for our application, where in $\catC_A^\loc$ the dual of the tensor unit is not the tensor unit, but otherwise well-behaved. The main example is if $\catC_A^\loc$ admits a Grothendieck-Verdier duality.
\begin{corollary}\label{cor_RidigityOfUA}\marginpar{ask TC}
Let $A$ be a commutative algebra $A$ in a braided tensor category,
which is simple as an $A$-module, and there exists an $A$-module $\overline{A^*}$ with a nonzero $A$-module morphism $A^*\otimes_A \overline{A^*}\to A$. Then $\UU_A$ and $\UU_A^\loc$ are also braided tensor categories  
\end{corollary}
\begin{proof}
    Concatenation gives a nonzero $A$-module morphism 
    $$A^*\otimes \overline{A^*}\to A^*\otimes_A \overline{A^*}\to A.$$
    Because $A$ is assumed to be a simple $A$-module, this map is surjective. Then the dual morphism in $\catC$ gives an embedding $A^*\hookrightarrow A\otimes X$ for \smash{$X=\overline{A^*}^*$}.
\end{proof}
\end{comment}

\subsection{Simple current extensions}\label{sec_simpleCurrent}

\begin{definition}
A \emph{simple current} in a monoidal category $\cat$ is an invertible object $C$, that is, there exists an object $C^{-1}$ with $C\otimes C^{-1}\cong C^{-1}\otimes C=\1$. The simple currents in $\cat$ form a group. 

A \emph{simple current extension} is a commutative algebra $A$ in a braided monoidal category, which has the form $A=\bigoplus_{i\in I} C_i$ for a subgroup $I$ of simple currents, with multiplication $C_i\otimes C_j\to C_{i+j}$ nonzero for any $i,j\in I$.
\end{definition}

In particular, a simple current is rigid even if the category $\cat$ is not known to be rigid. The following nice properties of simple current extension will be crucial. We adapt the following  from \cite{CMY21} Proposition 3.2.4:

\begin{lemma}
For a simple current extension $A$ the induction functor $A\otimes(-):\,\cat\to \catA$ is exact. 
\end{lemma}
Moreover, it is not difficult to see
\begin{lemma}\label{lm_inductionIndecomposable}
For a simple current extension $A=\bigoplus_{i\in I}C_i$ with $C_0=\1$, assume it is \emph{fix-point-free} in the sense that $C_i\otimes X\not\cong X$ for any $i\neq 0$ and any simple object $X$. Then the induction functor $A\otimes(-):\,\cat\to \catA$ sends simple modules to simple modules and indecomposable extensions of simple modules to indecomposable extensions of simple modules.
\footnote{We thank T. Creuztig for this remark. For iterated extensions, it seems fix-point-free has to be replaced by torsion-free. On the other hand, we only require later-on that the image contains non-simple indecomposables, which is probably true without such assumptions.}
\end{lemma}
\begin{proof}
Let $X\in\cat$ be simple, then $A\otimes X=\bigoplus_{i\in I} C_i\otimes X$ is a direct sum of simple objects in $\cat$, which are by assumptions mutually nonisomorphic, because $C_i\otimes X\cong C_j\otimes X$ implies $C_{i-j}\otimes X\cong X$.  Hence, any nonzero submodule of $A\otimes X$ in $\cat$ contains at least one of these summands. Because we assumed the multiplication is nonzero, in $\catA$ any simple submodule of $A\otimes X$ is already all of $A\otimes X$. This proves that induction preserves simplicity.

Let $X\stackrel{\iota}{\to} M\stackrel{\pi}{\to}  Y$ be a nonsplit extension of simple object in $\cat$. Then by exactness of $A\otimes(-)$ we have have an extension in $\catA$
$$\bigoplus_i X \otimes C_i\to \bigoplus_j M\otimes C_j\leftrightarrows \bigoplus_k Y\otimes C_k.$$
Suppose now there is a section $s$ in $\cat_A$, contrary to the assertion. We consider this as a split exact sequence in $\cat$: The summands in the second sum are nonsplit exact extensions of simple objects  $X\otimes C_j\to M\otimes C_j\to Y\otimes C_j$. Hence the space of $\cat$-morphism $s$ between any summands in the third and second sum is at most $1$-dimensional, namely $(\iota\otimes \id_{C_j})\circ f:\,M\otimes C_j\rightarrow Y\otimes C_k$ in case there is for $j,k$ an isomorphism of simple objects $f:X\otimes C_j\rightarrow Y\otimes C_k$. But the image is in the kernel of $\pi\otimes \id_A$, hence from  such $s$ we can never built a section of $\pi\otimes \id_A$.
\end{proof}

\begin{example}
We discuss simple current extensions of quadratic spaces $\Vect_\Gamma^Q$: This should probably be attributed back to \cite{DL93} in the vertex algebra context, it is well known, and the interested reader maybe also referred to our thorough discussion in \cite{GLM24} Section 4.2:
Let $I\subset \Gamma$ be an isotropic subgroup, meaning $Q(i)=1$ for all $i\in I$ and in consequence $B(i,j)=1$ for $i,j\in I$. Then we can define a commutative algebra in $\Vect_\Gamma^Q$
$$A=\C_\epsilon[I]=\bigoplus_{i\in I} \C_i$$
with multiplication $\C_i\otimes \C_j \stackrel{\epsilon(i,j)}{\mapsto} \C_{i+j}$ for a $2$-cocycle $\epsilon$ such that 
$\epsilon(i,j)\epsilon(j,i)^{-1}=\sigma(i,j)$, which exists since $\sigma(i,j)\sigma(j,i)=1$ and $\sigma(i,i)=1$. Then 
\begin{align*}
\left(\Vect_{\Gamma}^Q\right)_A
&=\Vect_{\Gamma/I}^Q\,,
\\
\left(\Vect_{\Gamma}^Q\right)_A^\loc
&=\Vect_{I^\perp/I}^{Q_I}\,,
\end{align*}
where $I^\perp$ is the set of all $v\in \Gamma$ with $B(v,i)=1$ for all $i\in I$ and $Q_I$ is the quadratic form $Q$ restricted to $I^\perp$ and factored over $I$. 

A particular example, relevant to our application, is again the infinite group $\Gamma=\R^n$ with and quadratic form $Q(v)=e^{\pi\i(v,v)}$. Then, isotropic subgroups are even integral lattice $\Lambda\subset \R^n$, and the local modules give the modular monoidal category $\Vect_{\Lambda^*/\Lambda}^{Q_I}$ related to the discriminant form of the lattice $(\Lambda^*/\Lambda,Q_I)$.
\end{example}

\subsection{Central functors and algebras}\label{sec_centralalgebra}

Recall that for a monoidal category $\BB$ the Drinfeld center $\cZ(\BB)$ is a braided monoidal category consisting of pairs of objects $X\in\BB$ and half-braidings $c_{X,-}$. We use the following suggestive language, which has become common:

\begin{definition}
A \emph{central tensor functor} $F:\UU\to \BB$ between a braided monoidal category $\UU$ and a monoidal category $\BB$ is a tensor functor together with an upgrade to a braided tensor functor $\tilde{F}:\UU\to \cZ(\BB)$ in the sense that the following diagram commutes
$$
\begin{tikzcd}
&&
\cZ(\BB)\arrow{d}{\mathrm{forget}\;c_{X,-}}
\\
\UU 
\arrow{rr}{F}
\arrow{rru}{\tilde{F}}
&&
\BB
\end{tikzcd}
$$
\end{definition}
As a special case we have the notion of a \emph{central (full) subcategory}.\\

\begin{definition} For a monoidal category $\BB$, a \emph{central commutative algebra} is a commutative algebra in the Drinfeld center $(A,c_{A,-}) \in \cZ(\BB)$. 

In a straightforward way, the constructions in Section \ref{sec_commAlgebras} carry over, that is: The category $\BB_{(A,c)}$ of representations of a central commutative algebra in $\BB$ with tensor product $\otimes_A$ is a monoidal category and the monoidal subcategory of local modules $\BB_{(A,c)}^\loc$ is a braided monoidal category.  
\end{definition}

Of course, if $\BB$ is braided, then every commutative algebra $A$ in $\BB$ becomes a central commutative algebra in the monoidal category $\BB$ by using the the braiding present in $\BB$.\\

\subsection{Frobenius-Perron dimension}\label{sec_FP}

Recall from \cite{EGNO15} Section 3.3 that for any unital transitive $\Z^+$-ring of finite rank we have a notion of the \emph{Frobenius-Perron dimension} $\FPdim(X)$ of an object, defined as the largest non-negative eigenvalue of multiplication with~$X$. Recall from \cite{EGNO15} Proposition 4.5.4 that if $\catC$ is a tensor category (that is: finite and rigid), then the Grothendieck ring $\mathrm{K}_0(\catC)$ fulfills the necessary properties. Recall from \cite{EGNO15} Definition 6.1.6 the \emph{regular object}
$$R_\catC=\sum_{i\in I} \FPdim(X_i)P_i,$$
where $X_i,P_i$ run over all simple objects $X_i$ and their projective covers $P_i$, and the \emph{global Frobenius-Perron dimension} is defined as $\FPdim(\catC):=\FPdim(R_\catC)$. We require in this article the following properties: \\

Let $F:\catC\to \catD$ be a tensor functor between tensor categories:
\begin{lemma}[\cite{EGNO15} Proposition 4.5.7]
For every object $X$ we have the equality $\FPdim(F(X))=\FPdim(X)$.
\end{lemma}
\begin{lemma}[\cite{EGNO15} Proposition 6.3.3 and 6.3.4]\label{lm_EGNOinjectivesurjective}
If $F$ is injective (that is: fully faithful), then $\FPdim(\catC)\leq \FPdim(\catD)$ and equality implies that $F$ is an equivalence of categories. 

If $F$ is surjective (that is: every simple $X\in\catD$ is a subquotient of some $F(Y)$), then $\FPdim(\catC)\geq \FPdim(\catD)$ and equality implies that $F$ is an equivalence of categories.  
\end{lemma}
More precisely, one can quantify
\begin{lemma}[\cite{EGNO15} Proposition 6.2.4]
If $F$ is surjective and $G$ a right adjoint, then for every object $X$
$$\FPdim(G(X))=\frac{\FPdim(\catC)}{\FPdim(\catD)}\FPdim(X)$$
\end{lemma}
In particular, in the situation of Lemma \ref{sec_commAlgebras} we get for $X=\1$
\begin{corollary}\label{cor_FP_CA}
    For $\cat$ a  braided tensor category and $\cat_A$ a tensor category we have
    $$\FPdim(\catC_A)=\frac{\FPdim(\catC)}{\FPdim(A)}.$$
\end{corollary}

On the other hand, for local modules we have 

\begin{lemma}[\cite{LW23} Corollary 4.26, \cite{SY24} Lemma 5.17]\label{cor_FP_CAloc}
 For $\catC$ a finite tensor category and $\cat_A$ a tensor category we have 
$$\FPdim(\catC_A^\loc)\geq \frac{\FPdim(\catC)}{\FPdim(A)^2},$$
and equality holds iff $\catC$ has a nondegenerate braiding.
\end{lemma}
Note that the converse statements holds, because equality means that the Schauenburg functor in the next section is injective and then our calculation shows it is also surjective.

\section{Relative Drinfeld centers and  reconstruction results}
\label{sec_reconstruction}

\subsection{Relative Drinfeld center}

For $\CC \subset \BB$ a central (full) subcategory, there is the notion of a relative Drinfeld center $\cZ_\CC(\BB)$, a braided monoidal subcategory of  $\cZ(\BB)$ defined as the the Müger centralizer of $\CC^{\mathrm{rev}}$ embedded into $\CC$ via the central structure, see \cites{Lau20,LW22}. Hence we have embeddings of braided monoidal categories

$$\CC \subset \cZ_\CC(\BB)\subset \cZ(\BB).$$

\subsection{Schauenburg functor}

We recall our following result from \cite{CLR23} Section 3, building on respective work in \cite{Sch01}.

\begin{definition}
Let $A$ be a commutative algebra in a braided monoidal category~$\UU$. Then there is a braided tensor functor, called \emph{Schauenburg functor} in what follows, to the relative Drinfeld center of $\UU_A$ over $\CC=\UU_A^\loc$
\begin{align*}
\Schauenburg:  \UU 
&\hookrightarrow \mathcal{Z}_{\CC}(\UU_A)\\
X 
&\mapsto (A\otimes X, c_{-, X}). 
\end{align*}
\end{definition}

That is, the induction functor is in fact a central functor compatible with the central subcategory $\CC\subset \UU_A$. The main point is now that in good cases this is an equivalence of braided tensor categories.

\begin{theorem}[\cite{CLR23} Corollary 3.8]\label{thm_Schauenburg}
Assume that $\UU,\UU_A$ are rigid and finite and $\UU$ has trivial Müger center and  $A$ is haploid. Then the Schauenburg functor is an equivalence of braided monoidal categories
\begin{align*}
\Schauenburg:  \UU 
&\cong \mathcal{Z}_{\UU_A^\loc}(\UU_A).
\end{align*}
\end{theorem}

For the proof we first establish that $\Schauenburg$ is fully faithful and then check that Frobenius-Perron dimensions on both sides coincide and apply Lemma \ref{lm_EGNOinjectivesurjective}. 

In fact, several weaker version of this statement can be proven. In \cite{CLR23} Lemma 3.4 we establish fully faithulness in locally finite cases and in \cite{CLR23} Corollary 3.13 we prove the equivalence in a ''relatively finite'' situation. The intention of these versions is to manage the infinite examples of unrolled quantum group and singlet vertex algebra. \\

In \cite{CMSY24} Theorem 3.21 the assumption of rigidity is largely removed and the nondegeneracy of the braiding is replaced by a version involving only simple modules contained in the algebra.

\subsection{Adjoint algebra}
%%%%%%%%%

%As a first application (and the one mainly indended by some authors mentioned above) one consideres the forgetful functor $F:\cZ_\CC(\BB)\to\BB$, which is faithfully exact, hence for finite categories admits a right adjoint and monadicity. 
Conversely, we can consider the right adjoint to the tensor functor $\cZ_\CC(\BB)\to \BB$ forgetting the half-braiding. The right adjoint is a lax tensor functor, so  the image of $\1$ is an algebra $A$ in $\cZ_\CC(\BB)$ called \emph{adjoint algebra}. By \cite{LW22} and \cite{SY24} Section 6.2.3 this algebra reconstructs our initial situation:

\begin{corollary}
Suppose $\BB$ is a tensor category (so: finite and rigid) and $\CC$ is a braided tensor  category with nondegenerate braiding, then the adjoint algebra~$A$ in $\cZ_\CC(\BB)$ has the property
$$\cZ_\CC(\BB)_A=\BB$$
$$\cZ_\CC(\BB)_A^\loc=\CC$$
\end{corollary}

In the more general case of a module category, see also \cite{BMM24} for the relative coend.

\begin{problem}
These assertions would be very useful to have for suitable algebras in more general braided monoidal categories. Note that \cite{LW22} give a version where $\BB$ consists of modules of a Hopf algebra and which  does not require finiteness. 
\end{problem}

\subsection{Generalized quantum groups}\label{sec_generalizedQuantumGroup}

Quantum groups will appear in the present article in the following form, as discussed in detail in \cite{CLR23} Section 6: Let $\Nichols$ be a Hopf algebra in a rigid braided monoidal category $\CC$, then we have basically by construction a rigid monoidal category of modules 
$$\Mod(\Nichols)(\CC),\;\otimes_\CC.$$
We can take the relative Drinfeld center of this monoidal category. For finite categories, this coincides with the category of $\Nichols$-Yetter-Drinfeld modules in $\CC$
$$\UU=\YD{\Nichols}(\CC)=\cZ_\CC(\Mod(\Nichols)(\CC)).$$
It is nondegenerately braided iff $\CC$ is nondegenerately braided.

In particular, this construction produces the category of representations of the quantum group. This has been present in some formulation since the beginning of quantum groups \cites{Drin86, Maj98, AG03}, we use the formulation obtained in \cites{Lau20,LW22}:

\begin{example}[Quantum groups]
Let $\g$ be a semisimple finite-dimensional complex Lie algebra and denote the simple roots by $\alpha_1,\ldots,\alpha_n$ and the Killing form by $(-,-)$. Let $\CC=\Vect_\Gamma^Q$ be the braided tensor category in Example \ref{ex_quadraticSpace} for $\Gamma$ some choice of a group containing elements $\alpha_1,\ldots,\alpha_n$ and $Q$ a quadratic form with $(\sigma,\omega)$ a choice of abelian $3$-cocycle such that $\omega=1$ and $\sigma(\alpha_i,\alpha_j)=q^{(\alpha_i,\alpha_j)}$ for a choice of scalar $q$. Let $X=\bigoplus_{i=1}^n \C_{\alpha_i}$ be an object on $\CC$ and $\NicholsOf(X)$ the Nichols algebra associated to the braiding $\sigma(\alpha_i,\alpha_j)$, as briefly reviewed in Section \ref{sec_NicholsDef}.

\medskip
Algebraically, we have the following: The dual group ring $\C^\Gamma$, which is a quasi-triangular Hopf algebra using $\sigma$,  is  the Cartan part of the quantum group $u_q(\g)^0$ with the $R$-matrix $R_0$, and  $\NicholsOf(X)$ is the positive part of the small quantum group $u_q(\g)^+$, and the smash product $u_q(\g)^0\ltimes \NicholsOf(X)$ is the Borel part $u_q(\g)^{\geq0}$. 

\medskip
Categorically, we have the following:
\begin{itemize}
\item The braided tensor category $\CC$ is the category of representations of $u_q(\g)^0$, the category of weight-spaces.
\item The tensor category $\BB=\Mod(\Nichols)(\CC)$ is the category of representations of the Borel part $u_q(\g)^\geq 0$ inside $\CC$, by the key property of the smash product.
\item We have an equivalence of  braided tensor categories, see  \cite{Lau20} Cor.~4.10:
$$\YD{\Nichols}(\CC)=\cZ_\CC(\BB)=\Mod(u_q(\g)).$$
\end{itemize}
Under this correspondence, the functor 
$$\cZ_\CC(\Mod(\Nichols)(\CC))\to \Mod(\Nichols)(\CC)$$
corresponds to restriction to the Hopf subalgebra 
$$u_q(\g)\hookrightarrow u_q(\g)^{\geq0}.$$
Accordingly, the left adjoint is the induction functor $u_q(\g)\otimes_{u_q(\g)^{\geq0}}(-)$ and the right adjoint is the coinduction functor. In particular the adjoint algebra is the dual Verma module of weight zero.
$$A=\big(u_q(\g)\otimes_{u_q(\g)^{\geq0}}\C_\epsilon\big)^*.$$
\end{example}

\begin{example}[Quasi quantum group] Suppose now in the situation above that $\omega\neq 1$ i.e. there is a nontrivial associator in $\CC$. For example, this is necessary if $R_0$ should be a nondegenerate braiding and the order of $q$ is even. Then we can take $\cZ_\CC(\Mod(\Nichols)(\CC))$ as definition of a braided tensor category with associator. We would strongly expect that it coincides with the respective constructions in \cites{CGR20, GLO18, Neg21}.
\end{example}

\begin{remark}
The reader is advised that the duality induced from $\CC$ may not be the same as the duality induced from $\UU$. In particular, the dual Verma module is not self-dual.  
\end{remark}

\section{Reconstructing a categorical Hopf algebra}\label{sec_splitting}

\subsection{Realizing Hopf algebra}

We recall a Tannaka-Krein type reconstruction result relative to a subcategory from our work \cite{LM24}

\begin{theorem}\label{thm_LM}
Let $\BB$ be a tensor category (that is: finite and rigid) and $\CC$ be a braided tensor category. Assume there is a  central functor $\iota:\CC\to \BB$ and a left inverse tensor functor $\l:\BB\to \CC$, which is exact and faithful. Then there exists a Hopf algebra $\Nichols$ in $\CC$, such that $\BB=\Comod(\Nichols)(\CC)$ and $\l$ is the functor forgetting the action of $\Nichols$ and $\iota$ is the functor endowing an object in $\CC$ with the trivial coaction via $\eta_\Nichols$.
\end{theorem}

\begin{remark}\label{rem_alternativeRigidity1}
We could drop the assumption of rigidity and obtain a bialgebra $\Nichols$. Note that one needs to additionally require the action of $\CC$ via $\iota$ on $\BB$ to be exact. Since later Theorem \ref{thm_CYrigid} will ensure rigidity, we will forgo this discussion.
\end{remark}

\begin{example}
Note that if $\CC$ is itself a category of representations of a finite-dimen\-sional Hopf algebra in vector spaces, this assertion is effectively the Radford projection theorem, see \cite{CLR23} Example 4.10. 
\end{example}

\subsection{Splitting subcategory}

\begin{definition}\label{def_splittingSubcat}
Let $\BB$ be a  monoidal category and $\CC$ be a monoidal full abelian subcategory (that is, the embedding is exact). 
%Weibel Section 1.2.
Then we define the abelian subcategory $\BB^\split$ as consisting of all objects in $\BB$ whose composition series has finite length and consists solely of objects in $\CC$. It is the smallest Serre subcategory containing $\CC$.
\end{definition}
\begin{lemma}\label{lm_splittingSubcat}
$\BB^\split$ is a monoidal subcategory.
\end{lemma}
\begin{proof}
   We prove this inductively. For two objects $X,Y\in \CC$  the tensor product $X\otimes Y$ is by assumption again in $\CC$. Let now $X_1\to X\to X_2$ be an extension of objects in $\BB^{\split}$ and $Y$ an object in $\BB^{\split}$ and assume inductively that already $X_1\otimes Y$ and $X_2\otimes Y$ are in $\BB^{\split}$. Then we have an exact sequence
   $$
   \cdots
   \longrightarrow
   X_1\otimes Y 
   \stackrel{i}{\longrightarrow} 
   X\otimes Y
   \stackrel{p}{\longrightarrow} 
   X_2\otimes Y
   \longrightarrow
    \cdots,
    $$   
    with possibly nontrivial terms $\cdots$ if the tensor product is not exact. Because we have assumed $\CC$ to be a full \emph{abelian} subcategory, 
     with  $X_1\otimes Y$ and $X_2\otimes Y$ also $\im(i)$ and $\ker(p)$ have composition series in $\CC$ and are hence in $\BB^{\split}$. Then we have a  short exact sequence in $\BB^{\split}$ 
    $$
    0
    \longrightarrow
    \im(i)
    \longrightarrow
    X\otimes Y
    \longrightarrow
    \ker(p)
    \longrightarrow  
    0.
    $$
    This proves inductively that also $X\otimes Y$ is in $\BB^{\split}$ as asserted.
\end{proof}

\begin{example}
If $\BB$ is semisimple, then $\BB^{\split}=\CC$. On the other hand if $\BB=\Mod(u_q(\g)^{\geq 0})$ and $\BB=\Mod(u_q(\g)^0)$, then  $\BB^{\split}=\BB$. The former is the typical setting for finite group orbifolds, the latter is the general setting in \cite{CLR23}.
\end{example}

\begin{remark}
There are arguments ensuring $\BB^{\split}=\BB$, for example if the category is graded and sufficiently unrolled. See \cite{CLR23} Definition 5.5 and the interesting idea in \cite{CN24} Section 5.5.
\end{remark}
Even under weaker duality assumptions on $\BB$ we have rigidity for $\BB^\split$:
\begin{theorem}[\cite{CMSY24} Theorem 3.13]\label{thm_CYrigid}
If $\catD$ is a locally finite abelian Grothen\-dieck-Verdier category with dualizing object $K=1$, whose simple objects are rigid, then $\catD$ is rigid
\end{theorem}

\begin{corollary}
    If $\CC$ is semisimple and rigid and $\BB^\split$ is locally finite and a
Grothendieck-Verdier category with dualizing object $K = 1$, then $\BB^\split$ is rigid.
\end{corollary}

%\begin{theorem}[\cite{CY24} Theorem 3.14]
%    Let $\UU$ be a locally finite abelian ribbon Grothendieck-Verdier category and let $A$ be
%a commutative algebra in $\UU$ such that $\theta^2_A=\id_A$. If the abelian braided monoidal category $\UU_A^\loc$ is rigid and
%every simple object of $\UU_A$ is an object of $\UU_A^\loc$, then $\UU_A$ is rigid.
%\end{theorem}

We now apply the previous subsection to this particular situation 

\begin{lemma}\label{lm_splitLM}
Let $\CC$ be a braided tensor category, which is a central full abelian subcategory of a tensor category $\BB$, then there exists a Hopf algebra $\Nichols$ in $\CC$ and an equivalence of monoidal categories
$$\BB^{\split}\cong\Comod(\Nichols)(\CC)$$  
Here, the functor $\l:\BB^{\split}\to \CC$ corresponds to the functor forgetting the action of $\Nichols$ and the functor $\iota:\CC\to \BB^{\split}$ corresponds to the functor endowing an object in $\CC$ with trivial $\BB$ action via $\epsilon_\Nichols$.
\end{lemma}
\begin{proof}
    There is a tensor functor $\l:\BB^{\split}\to \CC$ given by sending a an object $X$ to the direct sum of its composition factors. The functor is $\l$ is exact and faithful, so we can apply Theorem \ref{thm_LM}.  
\end{proof}

\subsection{A Hochschild argument and the Nichols algebra of~ \texorpdfstring{$\Ext^1$}{Ext1}}\label{sec_NicholsExt}

Let $\BB=\Mod(\Nichols)(\CC)$ for $\Nichols$ a bialgebra in $\CC$ and we assume that all simple objects in $\BB$ are of the form $M_\epsilon$ with $M\in\CC$ simple and $\Nichols$ acting trivially via $\epsilon_\Nichols$. That is to say, in the language of the previous section $\BB=\BB^{\split}$. 

The goal of this section is to determine $\Nichols$ in terms of a Nichols algebra associated to $\Ext^1_\BB(\1,M_\epsilon)$ for the various simple objects $M$. \\

We start by discussing a version of first Hochschild cohomology and its relation to $\Ext^1$ in the category of modules over an augmented algebra $\Nichols$ inside a monoidal category $\CC$. It is completely analogous to the classical algebra statements: 
Let  $\Nichols, \eta,\mu$ be an algebra in a monoidal category $\CC$ and $\epsilon:\Nichols\to \1$ be a fixed algebra augmentation. Consider an extension of $\Nichols$-modules between the unit object $\1$ with trivial $\Nichols$-action $\epsilon$ and some other object $M$ with arbitrary $\Nichols$-action $\rho$
$$M_\rho\to E\to \1_\epsilon,$$
which we assume to be \emph{$\CC$-split}, that is, the underlying extension of $\CC$-objects splits as $E\cong M\oplus \1$. Accordingly, the action $\Nichols\otimes E\to E$ can be written in four components (''block matrices''), of which three are fixed by the assumptions.
\begin{align*}
\epsilon:\;& \Nichols\otimes \1\to \1, \\
0:\;& \Nichols\otimes M\to \1, \\
\d_E:\;& \Nichols\otimes \1\to M, \\
\rho:\;& \Nichols\otimes M\to M. 
\end{align*}
The remaining morphism $\d_E:\Nichols\to M$ inherits the following properties from the multiplicativity of the action
\begin{definition}
A derivation  with values in a $\Nichols$-module $M_\rho$ is a $\CC$-morphism $\d:\Nichols\to M$ such that 
$$\d\circ \eta =0,\qquad \d\circ \mu= \rho\circ(\id\otimes \d) + \d\otimes\epsilon$$
\end{definition}
Conversely, every such derivation defines an extension $E$. Next, an  equivalence of extensions $E,E'$ is an isomorphism  of $\Nichols$-modules $\Phi:E\cong E'$ such that the following diagram commutes \\
$$\begin{tikzcd}[row sep=10ex, column sep=15ex]
M_\rho 
\arrow{r}{} 
\arrow{d}{\id}
& 
E 
\arrow{r}{}
\arrow{d}{\Phi}
& 
\1_\epsilon
\arrow{d}{\id}
\\
M_\rho 
\arrow{r}{} 
& 
E'
\arrow{r}{}
& 
\1_\epsilon
\end{tikzcd}$$

Then again we write $\Phi$ in components, where the  only nontrivial component is some $\CC$-morphism $\varphi:\1\to M$, and $\Phi$ is a $\Nichols$-module morphism iff
$\d_E+\varphi\circ\epsilon=\d_{E'}+\rho \circ (\id_\Nichols\otimes\varphi)$, which means $$\d_E-\d_{E'}=\rho \circ (\id_\Nichols\otimes\varphi)- \varphi\circ \epsilon.$$ 
The difference could be called an \emph{inner derivation} $\partial\varphi$. Note that if $\1,M$ in $\CC$ are simple and nonisomorphic, then there are no nonzero inner derivations.  

\begin{lemma}
The vector space $\Ext^{1}_{\Mod(\Nichols),\,\CC\,\split}(\1_\epsilon,M_\rho)$ of extensions in the category of $\Nichols$-modules that split in $\CC$ is isomorphic to the vector space of derivations $\d:\Nichols\to M$ up to inner derivations $\partial\varphi$ for $\varphi:\1\to M$.
\end{lemma}

\begin{example}
    Let $\psi: \Nichols\to \1$ be another algebra morphism and $M_\psi$ be the corresponding $\Nichols$-module on the $\CC$-object $M$. Then a derivation $\d_E$ with values in $M_\psi$ is a $\CC$-morphism $\d_E:\1\to M$ such that  
    $$\d_E\circ \eta =0,\qquad \d_E\circ \mu= \psi\otimes \d_E+\d_E\otimes\epsilon,$$
    which could be called an \emph{$(\epsilon,\psi)$-derivation}
    an an inner derivation is $\partial\varphi=\varphi\circ(\psi-\epsilon)$ for $\varphi:\1\to M$.
\end{example}

 In particular for an extension by $\Nichols$-modules of the form $M_\epsilon$ for $M\in\CC$, called henceforth \emph{trivial $\Nichols$-action}, we get:

\begin{corollary}    
The vector space $\Ext^{1}_{\Mod(\Nichols),\,\CC\,\split}(\1_\epsilon,M_\epsilon)$ of $\CC$-split extensions between $\Nichols$-modules with trivial action is isomorphic to the vector space of $(\epsilon,\epsilon)$-derivations $\d:\Nichols\to M$. Note that there are no nonzero inner $(\epsilon,\epsilon)$-derivation.
\end{corollary}

We also record the dual version of these results, which is proven by reversing arrows: Let $\Nichols$ be a coalgebra in a monoidal category $\CC$. Let $\eta:\1\to \Nichols$ be a fixed coalgebra morphism and denote for $M\in\CC$ by $M_\eta$ the $\Nichols$-comodule with coaction given by $\eta$, called henceforth \emph{trivial $\Nichols$-coaction}.

\begin{lemma}\label{lm_CoExt}
 The vector space $\Ext^{1}_{\Comod(\Nichols),\,\CC\,\split}(M_\eta,\1_\eta)$ of $\CC$-split extensions between trivial $\Nichols$-comodules is isomorphic to the vector space $\Hom_\CC(M,\Nichols^{\prim})$, where  we define the $\CC$-subobject $\Nichols^{\prim}$ of $\Nichols$ as the equalizer of $\Delta_{\Nichols}$ and $\id\otimes \eta_\Nichols+\eta_\Nichols \otimes \id$.
\end{lemma}

\begin{example}
Let $\CC=\Vect$ and $\Nichols$ a coalgebra with a fixed grouplike $1$, then it is quite clear that primitive elements $x\in\Nichols$, that is $\Delta(x)=x\otimes 1 + 1\otimes x$, are in correspondence to self-extensions of the trivial comodule $\C\to E\to \C$. 

Let $\CC=\Vect_\Gamma$ and $\Nichols$ a coalgebra  with a fixed grouplike $1$, then again it is quite clear that primitive elements in $\Nichols$ in degree $g-h\in \Gamma$ are in correspondence to extensions $\C_g\to E\to \C_h$ of the $1$-dimensional objects with trivial coaction. In the smash-product $\C^\Gamma\ltimes\Nichols$ these become $(g,h)$-primitive elements. This and further assertions are known as the \emph{Taft-Wilson theorem}, see \cite{EGNO15} Corollary 1.13.6.
\end{example}

\begin{definition}\label{def_basic}
    An augmented algebra $(A,\epsilon)$ in $\CC$, whose simple modules are trivial modules $M_\epsilon$ could be called \emph{local} (in a different sense then above) or \emph{basic}. Dually, a coalgebra $C$ with a distinguished grouplike unit $\eta:\1\to C$, whose simple comodules are trivial comodules $M_\eta$ could be called \emph{connected}.
\end{definition}

\section{Nichols algebra arguments}\label{sec_NicholsArguments}

\subsection{Nichols algebra}\label{sec_NicholsDef}

 Let $M$ be an object in a braided monoidal category $\CC$. The tensor algebra $\mathfrak{T}(M)=\bigoplus_{k\geq 0} M^{\otimes k}$ becomes a Hopf algebra in $\CC$: We define the coproduct on $M$ by 
 $$\Delta:\;M\to \1\otimes M + M\otimes \1$$
 $$\Delta=\1\otimes \id+\id\otimes \1$$
 (where $\1\otimes (-)$ means the unitality constraint) and extend it uniquely to an algebra morphism on $\mathfrak{T}(M)$. The antipode is given on $M$ by $-\id$, and we extend it uniquely to an anti-algebra morphism on $\mathfrak{T}(M)$. 

\begin{definition}
    Let $M$ be an object in a braided monoidal category $\CC$. The \emph{Nichols algebra} $\NicholsOf(M)$ is the Hopf algebra quotient of $\mathfrak{T}(M)$ by the largest graded Hopf ideal in degree $\geq 2$.
\end{definition}
For any bialgebra $\Nichols$ in $\CC$ we define the \emph{primitive subspace} as an equalizer 
$$P(\Nichols)=\mathrm{Eq}\big(\Delta_{\Nichols},\;
\eta_{\Nichols}\otimes \id_{\Nichols} + \id_{\NicholsOf(M)}\otimes \eta_{\Nichols}\big)$$
By definition $P(\NicholsOf(M))\supset M$. In fact, for a graded Hopf algebra $\Nichols$ any homogeneous component in $P(\Nichols)$ is a subcoalgebra and thus generates a Hopf subalgebra. Conversely, for any graded Hopf ideal in $\Nichols$ the homogeneous component in lowest degree is contained in $P(\Nichols)$. Hence 
\begin{corollary}
    The Nichols algebra $\NicholsOf(M)$ is the universal Hopf algebra quotient of $\mathfrak{T}(M)$ such that $P(\NicholsOf(M))=M$.
\end{corollary}
For Nichols algebras in $\Vect_\Gamma^Q$ or more generally in the category of Yetter-Drinfeld modules over a Hopf algebra in $\Vect_{\mathbb{K}}$ there are several equivalent definitions of a Nichols algebra, see \cite{HS20}  Section 1.9. For example, it is the quotient by a nondegenerate Hopf pairing with $\NicholsOf(M^*)$ and it is the quotient of $\mathfrak{T}(M)$ by the kernel of the quantum symmetrizer map. For Hopf algebras in arbitrary braided tensor categories, the proofs for these equivalent characterization seem to carry over, but we refrain to do so in the present article.

\subsection{Two notions of connectivity}

For Nichols algebras in $\Vect_A^Q$, or more generally in the category of Yetter-Drinfeld modules over a Hopf algebra in $\Vect_{\mathbb{K}}$, as in \cite{HS20}, the notion of connectivity and the coradical filtration is defined in terms of a coalgebra in $\Vect_{\mathbb{K}}$, not in the actual category. To apply their results, we need to compare these definition (or dually the definitions of local algebras).

\begin{problem}\label{problem_connected}
 There is a rich theory of local and basic algebras, see for example \cite{Zimm14} Section 1.6, which in particular describes the radical filtration and Nakayama's lemma, and Section 4.5, which discusses in particular Gabriel's theorem to obtain a basic algebra as a quotient of a path algebra, where in our picture nodes correspond to simple objects and arrows to skew primitive elements and equivalently to elements in $\Ext^1$. It is conceivable that most of this theory should carry over to the categorical setting. Note that the coradical filtration in this context is obtained in \cite{EGNO15} Section 1.13 and a Jacobson radical has been developed recently in \cite{CSZ25}. 
\end{problem}

\newcommand{\Fla}{F^{\mathrm{la}}}

We need another version of Frobenius reciprocity: Let $F:\CC\to \CC'$ be a tensor functor and $\Fla$ its left adjoint. The adjunction gives a bijection of Hom-spaces
$$\Hom_\CC(\Fla(M),N)\cong\Hom_{\CC'}(M,F(N))$$
$$f\mapsto  F(f)\circ \eta^{F,\Fla}$$
Recall  that $\Fla$ is oplax, so it sends coalgebras to coalgebras and comodules to comodules, in particular if $M$ is a $C'$-comodule and $N$ is a $C$-module, then $\Fla(M)$ is a $\Fla(C')$-comodule and $F(N)$ is a $F(C)$-comodule. If we choose $C'=F(C)$, then the counit of the adjunction gives a coalgebra morphism $\Fla(C')=\Fla(F(C))\to C$, and thereby $\Fla(M)$ becomes a $C$-comodule. 

\begin{lemma}\label{lm_ConnectedFrobenius}
With these structures, the bijection of homomorphisms upgrades to a bijection of comodule homomorphisms
$$\Hom_{\Comod(C)}(\Fla(M),N)\cong\Hom_{\Comod(F(C))}(M,F(N))$$
\end{lemma}
\begin{proof}
The action of $F$ on Hom-spaces 
$$\Hom_\CC(\Fla(M),N)\stackrel{F}{\longrightarrow}\Hom_{\CC'}(F(\Fla(M)),F(N))$$
sends morphisms compatible with coactions $\smash{\epsilon^{F,\Fla}_{C'}\circ \Fla(\delta_M)}$ and $\delta_N$ to morphisms compatible with coactions $\smash{F(\epsilon^{F,\Fla}_{C'})\circ F(\Fla(\delta_M))}$ and $F(\delta_N)$. Precomposing with $\eta_M$ turns by naturality the $C'$ coaction on $M$ into the $F(\Fla(C'))$ coaction on $F(\Fla(M))$, and hence sends morphisms compatible with coactions $\smash{F(\epsilon^{F,\Fla}_{C'})\circ F(\Fla(\delta_M))}$ and $F(\delta_N)$ to morphisms compatible with coactions $\delta_M$ and $D(\delta_N)$. 
\end{proof}

With this preparation we can compare two notions of connectedness: We have in mind in particular the case of $\CC'=\Vect$ and $\CC=\Vect_\Gamma^Q$ or  Yetter-Drinfeld modules over a Hopf algebra.

\begin{lemma}\label{lm_fibreConnected}
Let $\CC$ be a locally finite monoidal category and $\Nichols$ in $\CC$ a coalgebra with distinguished grouplike unit $\eta:\1\to \Nichols$ that is connected in the sense of Definition \ref{def_basic}, that is, all simple $\Nichols$-comodules in $\CC$ are trivial $\Nichols$ comodules $M_\eta$ for simple $M\in\CC$. 

Now let $F:\CC\to\CC'$ be a faithful tensor functor admitting a left adjoint $\Fla$, then $F(\Nichols)$ is also connected as a coalgebra in $\CC'$. 
\end{lemma}
\begin{proof}
Let $(M,\delta)$ be a simple  $F(\Nichols)$-comodule in $\CC'$. Then as in the previous Lemma $\Fla(M)$ is a $\Fla(F(\Nichols))$-comodule and using the counit of the adjunction a $\Nichols$-comodule. 
Since $\CC$ is assumed to be locally finite, there exists a nonzero simple quotient comodule, and since $\Nichols$ is assumed to be connected in $\CC$ we know it is of the form $N_\eta$ for $N\in \CC$. We now apply Lemma \ref{lm_ConnectedFrobenius}
$$\Hom_{\Comod(\Nichols)}(\Fla(M),N_\eta)\cong\Hom_{\Comod(F(\Nichols))}(M,F(N)_\eta)$$
By construction of $N$ as a quotient we have an epimorphism in the left spaces. Correspondingly, we have a nonzero map in the right space. Since $M$ was assumed simple, there are two possibilities: $F(N)=0$, but a left-exact faithful functor $F$ reflects zero-objects, so $N=0$ in contradiction to  its construction. Or $F(N)_\eta \cong M$, so $M$ has trivial coaction as asserted.  
\end{proof}

%%%% Hopf algebra proof
\begin{comment}
\begin{lemma}
Assume that $\CC=\Mod(C)$ is the category of representations of a Hopf algebra. Then the image under the fibre functor $f(\Nichols)$ is a connected coalgebra in $\Vect_\mathbb{K}$.
\end{lemma}
\begin{proof}
There is the obvious and well-known equivalence of the linear categories between modules over $\Nichols$ as algebra in $\Mod(C),\otimes_\C$ and modules over the smash product $H:=\Nichols\rtimes C$.  
Let $M$ be a simple module over $\Nichols$ in $\Vect$. We consider $\Nichols$ as a subalgebra of $H$ and take the induced module 
$$\tilde{M}=\Ind^H_\Nichols(M),$$
which can be considered as a module over $\Nichols$ in $\CC$. By the ascending chain condition, it has a simple quotient module, which by assumption of $\Nichols$ local in $\CC$ has to be a $\Nichols$-module $S_\epsilon$. We now use Frobenius reciprocity 
$$\Hom_\Nichols(M,S_\epsilon)=\Hom_H(\tilde{M},S_\epsilon)$$
Since by construction we have a nonzero element in the right space, we also have a nonzero element in the left space. Since we assumed $M$ to be simple, it has to be the trivial  $\Nichols$-module.
\end{proof}
\end{comment}

\subsection{Liftings and generation in degree 1}

Nichols algebras are by construction connected Hopf algebras. In the situations of a diagonal braiding, where we understand the Nichols algebras well, the converse can be proven to be also also true. This question and subsequent answers are a main content of the Andruskiewitsch-Schneider program \cite{AS10} and the lifting method \cite{AG19}. We quote and then adapt the respective statements. The first asserts that there are no nontrivial deformations or \emph{liftings}:

\begin{theorem}[\cite{AKM15} Theorem 6.4]\label{thm_lifting}
If $\Nichols$ is a finite-dimensional bialgebra in $\Gamma$-Yetter-Drinfeld modules with $\Gamma$ abelian, whose associated graded bialgebra, with respect to the coradical filtration, is equivalent to a Nichols algebra $\NicholsOf(M)$, then  $\Nichols\cong \NicholsOf(M)$.
\end{theorem}
We briefly discuss their proof strategy: In \cite{MW09} Lemma 4.2.2 the authors discuss the possible lifting $\Nichols$ as deformations by Hochschild cycles and show their existence requires a nonzero morphism (in the category) between $M$ and the defining relations $R$. In our case, let us denote by $\C_g^\chi$ the simple $\Gamma$-Yetter-Drinfeld modules and $M=\bigoplus_t \C_{g_t}^{\chi_t}$, then there cannot be  liftings if $(g_R,\chi_R)\neq (g_t,\chi_t)$ for all relations $R$ and all $t$. In \cite{AKM15} Proposition 6.2 the authors go through the classification of Nichols algebras of diagonal type \cite{Heck09} and the presentation in terms of generators and relation \cite{An13} and show that this condition holds.

\begin{remark}
One typically also asks for all possible liftings of the smash product $\C[\Gamma]\#\NicholsOf(M)$ and here the answer is much more rich, see again \cite{AG19}. However, the observation is that the lifting is always a Doi-cocycle twist, meaning that the categories of comodules is equivalent to the category of comodules for trivial lifting $\Comod(\Nichols)(\CC)$. In the present article we have already established a categorical splitting, so we already are categorically in the situation with trivial lifting. We will discuss the interplay between these observations in a separate joint article.
\end{remark}

The quoted articles assume, in our present language, that $\Vect_\Gamma^Q$ can be realized with an abelian $2$-cocycle $(\omega,\sigma)$ with $\omega=1$. In the general case $\omega\neq 1$ we can follow  \cite{AG17} and find an extension $\tilde{\Gamma}\to \Gamma$. Note that it is not clear a-priori that an argument excluding nontrivial liftings for $\tilde{\Gamma}$ is also excluding nontrivial liftings for the quotient $\Gamma$, because we could have $(g_R,\chi_R)= (g_t,\chi_t)$ albeit $\tilde{g}_R\neq \tilde{g}_t$.  However, it follows from the equivariantization procedure in \cite{AG17} that this conclusion holds.\footnote{We thank Ivan Angiono for pointing this out.} Alternatively, one can check directly that the casewise arguments in \cite{AKM15} are based on braidings being different $\chi_R(g_R)\neq \chi_t(g_t)$, and these quantities are kept invariant by the extension.

\begin{corollary}\label{lm_myNoLifting}
If $\Nichols$ is a finite-dimensional bialgebra in $\Vect_\Gamma^Q$, whose associated graded bialgebra, with respect to the coradical filtration, is equivalent to a Nichols algebra $\NicholsOf(M)$, then  $\Nichols\cong \NicholsOf(M)$.
\end{corollary}
\begin{corollary}\label{lm_myNoLiftingInclusion}
If $\Nichols$ is a finite-dimensional  bialgebra in $\Vect_\Gamma^Q$ such that $P(\Nichols)=M$, then there is an injective morphism of bialgebras
$$\NicholsOf(M)\hookrightarrow \Nichols$$
\end{corollary}

We now want to see that the injective map in Lemma \ref{lm_myNoLiftingInclusion} is bijective. This was again settled by Angiono:

\begin{theorem}[\cite{An13} Theorem 4.15]\label{thm_GenDegOne}
    Let $H$ be a finite-dimensional Hopf algebra, whose coradical is an abelian group, then $H$ is generated by its grouplikes and skew-primitive. 
\end{theorem}

Again we review the proof strategy and make minor adaptions: Essentially this theorem is proven as a dualization of \cite{An13} Theorem 4.13, which states that a finite-dimensional graded Hopf algebra $\Nichols$ in $\Gamma$-Yetter-Drinfeld modules with homogeneous component $\Nichols_0=1$ and generated by $\Nichols_1$ is already the Nichols algebra - that is, there are no finite-dimensional extensions of Nichols algebras of diagonal type. For such an extension, one can look at a minimal new element, which would be primitive. Then essentially the same arguments as in Theorem \ref{thm_lifting}. 
In particular the theorem does not depend on $H$ having an antipode and $\Vect_\Gamma^Q$ instead of Yetter-Drinfeld modules. \\

By combining these Nichols algebra results with Lemma \ref{lm_CoExt}, the main result in the previous section, we arrive at the following statement that characterizes a bialgebra in terms of its $\Ext^1$:

\begin{theorem}\label{thm_ReconstructNichols}
    A connected bialgebra $\Nichols$ in $\CC=\Vect_\Gamma^Q$ is isomorphic to the Nichols algebra $\NicholsOf(X)$ of the object $X=\bigoplus_{a\in\Gamma} \Ext^1_{\Comod(\Nichols)(\CC)}(\C_a,1) \C_a$. 
\end{theorem}

By further applying this to the bialgebra realizing $\BB^\split$ by our Tannaka-Krein reconstruction in Section \ref{sec_splitting}, we obtain the main result in this section: 

\begin{theorem}\label{thm_ReconstructNichols2}
        Let $\BB$ be a tensor category and $\CC$ a central monoidal full subcategory, which is of the form $\CC=\Vect_\Gamma^{Q}$. Then the monoidal Serre subcategory $\BB^{\split}$ is equivalent to $\Comod(\Nichols)(\CC)$ where $\Nichols=\NicholsOf(X)$ is the Nichols algebra of the object 
        $X=\bigoplus_{a\in\Gamma} \Ext^1_{\BB}(\C_a,\1) \C_a$.
\end{theorem}

\begin{remark}
Note again that the results in this section  do not require $\Nichols$ to be a Hopf algebra, so again rigidity is \emph{not} used. Indeed, a connected bialgebra is always a Hopf algebra by \cite{Mon93} Lemma 5.2.10 and Corollary 5.2.11, which goes back to Takeuchi. 
\end{remark}

\section{Proving Kazhdan-Lusztig using Frobenius-Perron dimension}\label{sec_proof1}

From the technology developed so far, we can derive a first algebraic proof of the logarithmic Kazhdan Lusztig conjecture, if we assume that sufficient structure is known on the category of representations of the vertex algebras. We recall the categorical setup:\\

Let $\UU$ be a braided tensor category (that is: finite and rigid). Let $A$ be a commutative algebra in $\UU$ and $\BB:=\UU_A$ and $\CC:=\UU_A^\loc$ the corresponding categories of modules and local modules. Let $\BB^\split \subset \BB$ as in Definition \ref{def_splittingSubcat}, then by Lemma \ref{lm_splitLM} there exists a Hopf algebra $\Nichols\in\CC$ such that 
$$\BB^\split\cong \Comod(\Nichols)(\CC).$$

\subsection{Vertex algebra setup}\label{sec_VOA}

Our application is the following situation: Let $\CC=\RepVOA(\V)$ be the category of representations of a suitable vertex operator algebra $\V$. Let $\W\subset \V$ be a conformal embedding of a suitable vertex operator algebra with $\UU=\RepVOA(\W)$. Then $A=\V$ can be considered as a commutative algebra in $\RepVOA(\W)$ and $\CC=\UU_A^\loc$, see \cites{HKL15, CKM24}.\\

We now assume that $\V=\V_\Lambda$ is a lattice vertex algebra associated to a even integral lattice $\Lambda\subset \R^n$, which we assume to be positive-definite and full-rank~$n$, so that the braided tensor category $\RepVOA(\V_\Lambda)$ is a  braided monoidal category, namely $\Vect_{\Gamma}^Q$ for $\Gamma=\Lambda^*/\Lambda$ the discriminant form, see Example~\ref{ex_quadraticSpace}. We denote by $\V_\lambda$ for $\lambda\in\Lambda^*$ the modules corresponding to $\C_\lambda$. As described for example in \cite{Len21}, let $\alpha_1,\ldots,\alpha_n$ be vectors in $\Lambda^*$ and choose a conformal structure\footnote{With such a modified conformal structure, $\RepVOA(\V_\Lambda)$ is a finite braided monoidal category with a Grothendieck-Verdier structure.} such that the conformal dimension $h(\alpha_i)=1$. Let $\zem_1,\ldots,\zem_n$ be the corresponding screening operators
$$\zem_i:\; \V_\lambda\to \V_{\lambda+\alpha_i}.$$
We define $\W\subset \V$ to be the intersection of the kernels of these screening operators $\zem_1,\ldots,\zem_n$ on $\V$. \\

As a particular class of examples, let $\g$ be a semisimple finite-dimensional complex Lie algebra with simple roots $\alpha_1,\ldots,\alpha_n$ and Killing form $(-,-)$. For a positive integer $p\geq 2$ divisible by the lacity of $\g$ let $\Lambda=\sqrt{p}\Lambda^\vee$ be the rescaled coroot lattice, then the kernel of screenings $\W=\W_{p}(\g)$ is the Feigin-Tipunin vertex algebra \cite{FT10}, see \cite{CLR23} and references therein. The braiding matrix is $q_{ij}=q^{(\alpha_i,\alpha_j)}$ for $q=e^{\pi\i/p}$. 
Then the Nichols algebra associated to this braiding $\NicholsOf(q)$ has generators $x_1,\ldots,x_n$ and fulfills quantum Serre relations. It also is the algebra generated by the screening operators $\zem_1,\ldots,\zem_n$.\footnote{Note that there is some strange behavior for very small values of $q$, namely $q^2=1$ for $ADE$ and $q^4=1$ for $BCF$ and $q^6=1$ and $q^4=1$ for $G_2$, see \cite{Len16} and in some other cases the Nichols algebra has relations in addition to the quantum Serre relations, for example $q^2=-1$ for $ADE$, see the case by case discussions in \cite{AA17}} The categories $\Mod(\NicholsOf(q))(\Vect_\Gamma^Q)$ and $\cZ_\CC(\Mod(\NicholsOf(q))$ are the representations of the Borel part and the full (quasi-)quantum group as discussed in Section \ref{sec_generalizedQuantumGroup}.

\begin{conjecture}[\emph{Logarithmic Kazhdan Lusztig correspondence}]
There is an equivalence of braided tensor categories 
$$\RepVOA(\W_p(\g))\cong \cZ_\CC(\Mod(\NicholsOf(q))).$$
The case $\sl_2$ is proven by \cites{FGST05, AM08, TW13, CGR20, CLR21, GN21, CLR23}.
\end{conjecture}

\subsection{Producing elements in \texorpdfstring{$\Ext^1$}{Ext1}}

Our goal is to identify several Nichols algebras: The Nichols algebra $\NicholsOf(q)$, which underlies the generalized quantum group for which we want to prove the  logarithmic Kazhdan-Lusztig conjecture above. It is isomorphic to the Nichols algebra $\NicholsScreenings$ of screening operators. We want to prove that it also coincides with the Nichols algebra $\Nichols$ determined by $\Ext^1(\UU_A)$ using the results in Section \ref{sec_NicholsArguments}, so that it is the Hopf algebra realizing $\UU_A$ over $\UU_A^\loc$. Do do so, we now introduce a tool to produce elements in $\Ext^1(\UU_A)$ from screenings:
\begin{lemma}\label{lm_produceExt}
Let $A_i\subset A$ be a subalgebra of $A\in\UU$ and denote $A$ considered as an $A_i$-module by $\bar{A}$. Then $(\UU^\loc_{A_i})_{\bar{A}}$ is a full tensor subcategory of $\UU_{A}=(\UU_{A_i})_{\bar{A}}$. In particular there is an injective map between the respective groups $\Ext^1(M,N)$. Note this is not necessarily true for higher $\Ext^i$.
%https://math.stackexchange.com/questions/3942663/injection-of-ext-groups 
\end{lemma}

This has the following application to kernel of screening vertex algebras, not just those associated to quantum groups. In particular it shows now that $\Nichols$ coincides with the Nichols algebras of screenings $\NicholsScreenings=\NicholsOf(q)$ for $q_{ij}=e^{\pi\i(\alpha_i,\alpha_j)}$.

\begin{lemma}\label{lem_KernelOfScreenings}
Let $\W=\bigcap_{i=0}^n\ker_{\zem_i}(\V_\Lambda)$ be the kernel of screening operators in $\V_\Lambda$. Assume that $(\alpha_i,\alpha_i)=2/p_i$ for some integer $p_i\geq 2$. Then there are nonzero elements $E_i\in\Ext^1_{\UU_A}(\C_{\alpha_i},\1)$ for $i=1,\ldots,n$ in the full tensor subcategory $(\UU^\loc_{A_i})_{{A}}$ for the subalgebra $A_i=\ker_{\zem_i}(\V_\Lambda)$ of $A$. 
\end{lemma}
\begin{proof}
We first consider the slightly smaller lattice $\Lambda_i:= p_i'\alpha_i\Z\oplus^\perp \Lambda_i'$ where $p_i'$ is some integer such that $\Z\alpha_i\cap \Lambda=p_i'\alpha_i\Z$ and $\Lambda_i'=\alpha_i^\perp\cap \Lambda$ is the orthogonal complement. This orthogonal sum $\Lambda_i$ has finite index in $\Lambda$, and we have a decomposition of the associated lattice vertex algebra
$$\V_{\Lambda_i}=\V_{p_i'\alpha_i^\vee\Z}\otimes_\C \V_{\Lambda_i'},$$
and of the corresponding categories of representations 
$$\RepVOA(\V_{\Lambda_i})=\RepVOA(\V_{p_i'\alpha_i\Z})\boxtimes \RepVOA(\V_{\Lambda_i'}).$$
Since the screening charge $\alpha_i$ is contained in the first factor and is by construction orthogonal to $\Lambda_i'$, the kernel of screenings has also a factored form  
$$\ker_{\zem_i}(\V_{\Lambda_i})
=\ker_{\zem_i}(\V_{p_i'\alpha_i\Z})
\otimes \V_{\Lambda_i'}.$$
The first factor is the triplet vertex algebra $\W_{p_i}(\sl_2)$, or possibly a version with a smaller or larger lattice if $p_i'\neq p_i$
$$\W_{p_i}^0(\sl_2) \subset \ker_{\zem_i}(\V_{p_i'\alpha_i\Z}) \subset \W_{p_i}(\sl_2) $$
From the proven logarithmic Kazhdan-Lusztig conjecture for $\sl_2$, even for the singlet vertex algebra $\W_{p_i}^0(\sl_2)$ corresponding to the unrolled quantum group, we know 
\begin{align*}
    \Ext^1_{\ker_{\zem_i}(\V_{p_i'\alpha_i\Z})}(\V_{\alpha_i},\1)&=\C^1, \\
    \Longrightarrow\qquad
\Ext^1_{\ker_{\zem_i}(\V_{\Lambda_i})}(\V_{\alpha_i},\1)&=\C^1.
\end{align*}
Now $\V_\Lambda \supset \V_{\Lambda_i}$ is a simple current extension, and correspondingly $\ker_{\zem_i}(\V_\Lambda) \supset \ker_{\zem_i}(\V_{\Lambda_i})$ is a simple current extension. For simple current extensions, induction sends simples to simples and  indecomposable extensions of simples to indecomposable extensions of simples by Lemma \ref{lm_inductionIndecomposable}, hence
\begin{align*}
    \C^1&\hookrightarrow\Ext^1_{\ker_{\zem_i}(\V_\Lambda)}(\V_{\alpha_i},\1).
\end{align*}
This concludes the claim. 
\end{proof}

Having a method to produce nonzero extension classes $E_i\in\Ext^1_{\UU}(\C_\alpha,\1)$, we have corresponding primitive elements in $\Nichols$ by Lemma \ref{lm_CoExt}, and by Theorem \ref{thm_ReconstructNichols2} we know: 

\begin{corollary}\label{cor_KernelOfScreeing}
In the situation of Lemma \ref{lem_KernelOfScreenings} we have an injection of Hopf algebras $\Vect_\Gamma^Q$  
$$\NicholsOf(q)\hookrightarrow \Nichols,$$
for $q_{ij}=e^{\pi\i(\alpha_i,\alpha_j)}$ with $1\leq i,j\leq n$. By construction, the primitive generators~$x_i$ correspond to the dual Hochschild cocycles corresponding to the extension $\1\to E_i\to \V_{\alpha_i}$.
\end{corollary}

Note that the assertion of Theorem \ref{thm_ReconstructNichols2} (that uses the knowledge of Nichols algebras of finite diagonal type) is stronger: It states that $\Nichols=\NicholsOf(X)$. However, we do not know if our previous construction covers all $\Ext^1(\1,M)$, and we do not know if $\BB=\BB^{\split}$. For example, we have not used that $\W$ is not smaller than the kernel of screenings. 

\subsection{Main proof}

Having constructed the correct Nichols algebra representation category inside $\UU_A$, it is tempting to conclude the proof by comparing Frobenius-Perron dimensions (see Section \ref{sec_FP}) of $\Nichols$ and our vertex algebra categories: 

\begin{lemma}\label{lm_matchBB}
Assume that $\UU,\UU_A$ are tensor categories and  $\BB'\subset \UU_A$ is a full monoidal subcategory and assume
$$\FPdim(\BB')=\FPdim(A)\FPdim(\UU_A^\loc).$$
Then $\BB'=\UU_A$ and $\UU_A$ is nondegenerate.
\end{lemma}
\begin{proof}
We recalled in Corollary \ref{cor_FP_CA} and Corollary \ref{cor_FP_CAloc} the Frobenius-Peron dimensions of $\UU_A$ and an inequality of $\UU_A^\loc$, hence with the assumption we get a chain of inequalities  
$$\FPdim(\UU_A)\geq \FPdim(\BB')
=\FPdim(A)\FPdim(\catC_A^\loc)\geq \frac{\FPdim(\catC)}{\FPdim(A)},$$
that is an overall equality. This implies by Lemma \ref{lm_EGNOinjectivesurjective} that $\UU_A=\BB'$ and by the last assertion in Corollary \ref{cor_FP_CAloc} that $\UU$ has a nondegenerate braiding. 
\end{proof}

As a consequence of this observation, our main result can be strengthened as follows:

\begin{theorem}\label{thm_KLviaFP}
Let $\W$ by a kernel of screening operators in a lattice vertex algebra $\V_\Lambda$ as in Lemma \ref{lem_KernelOfScreenings}. Assume that $\RepVOA(\W)$ is a braided tensor category. Assume the Frobenius-Perron dimension of $\V_\Lambda$ over $\W$ is equal to the dimension of the diagonal rank $n$ Nichols algebra $\NicholsOf(q)$ in Corollary \ref{cor_KernelOfScreeing}. Then there is an equivalence of tensor categories resp. braided tensor categories 
\begin{align*}
\UU_A &= \Mod(\NicholsOf(q))(\CC) \\
\UU&=\cZ_\CC(\Mod(\NicholsOf(q)))
\end{align*}
This is to say, $\UU$ is equivalent to representations over a generalized quantum group in the sense of Section \ref{sec_generalizedQuantumGroup}. 
\end{theorem}
\begin{proof}
We first note that $\UU_A$ is a tensor category by 
Corollary \ref{cor_isTensorCat}. %Alternative: using  Corollary \ref{cor_RidigityOfUA}, which requires $A$ to be simple and $A^*$ admitting a "dual" in $\catC_A^\loc$. In \cite{CMSY24} Corollary~4.2 it is checked that in $\UU_A^\loc=\RepVOA(\V)$ the vertex algebra dual (the contragradient dual) coincides with the dual for $A$-modules. Note that the lattice vertex algebra with the modified conformal structure necessary to define screening operators is usually not self-dual, so $A$ is typically not self-dual.

Now the assumption on $\FPdim(A)$ allows us to apply the previous Lemma~\ref{lm_matchBB} to show that $\BB=\BB^{\split}$, which is described by the Nichols algebra. This shows $\UU_A= \Mod(\Nichols)(\CC)$.

We then use that the Schauenburg functor 
is an equivalence by  Theorem~\ref{thm_Schauenburg}, which we can apply because Lemma \ref{lm_matchBB} also proved nondegeneracy of the braiding. This concludes $\UU\cong\cZ_\CC(\Mod(\Nichols(\UU))$. 
\end{proof}

\begin{remark}\label{rem_FPproof_whatnow}
It is not so clear how to get our hand on the Frobenius-Perron dimension without sufficient knowledge of the category $\UU$ and $\UU_A$, including the knowledge of some tensor products. The author can think of different strategies:
\begin{itemize}
    \item Prove that in $\UU_A$ all simple modules are local and $\Ext^1$ is as expected. This is roughly how we have proven Kazhdan-Lusztig for $\sl_2$ in \cite{CLR23}, namely from the knowledge of $\UU$ as abelian category and some fusion products and then by-hand computing $\UU_A$ from using induction and Frobenius reciprocity. The results in this article reduce the knowledge of $\UU_A$ as an abelian category to the knowledge of  $\Ext^1$. Note that this is the only case in which Theorem \ref{thm_GenDegOne} becomes fully effective (in contrast, all dimension arguments do not require to know generation in degree $1$ in advance).
    \medskip
    \item Compute $A\otimes_\W A$ and check that the composition series has $\dim\NicholsOf(q)$ many composition factors. In fact, $A\otimes_\W A$ corresponds  to $\Nichols$ as the regular $\Nichols$-module. The results in this article reduces the required knowledge from its structure as a module  to its mere size.
    \medskip
    \item In the next section we will use vertex algebra characters to compute the Frobenius-Perron dimension. This requires the assumption of rigidity and of Frobenius-Perron dimension being given by analytic expressions, for which the proof seems to depend on some Verlinde formula. 
    
    \medskip
    A much better approach in this direction would be to have a rigorous notion of dimension that can be computed and does not require rigidity (so it is in general maybe only submultiplicative). A good candidate would be the $C_1$-dimension, see \cite{MS23}. It would be an interesting problem to compute it for a kernel-of-screening vertex algebra.
    \medskip
    \item A very different method to analyze the difference between $\Nichols$ and the hypothetically larger category would be to compare the decomposition behavior of $A$ to that of the regular representation of $\Nichols$, and it is quite surprising that this gives information about the ambient category without prior knowledge of the ambient category outside of the one described by $\Nichols$. This is our goal for the next paper on this topic. 
\end{itemize}
\end{remark}

\newcommand{\M}{\mathcal{M}}
\newcommand{\dimVOA}{\chi}
\newcommand{\h}{\mathfrak{h}}
\newcommand{\qdimVOA}{\mathrm{qdim}}
\newcommand{\tr}{\mathrm{tr}}
\newcommand{\Hn}{\mathrm{H}}
\renewcommand{\q}{\mathrm{q}}

\section{Proving Kazhdan-Lusztig using asymptotics of characters}

\subsection{Flag variety and character formula}\label{sec_CharacterFormula}

We now introduce a very useful tool for vertex algebras that goes beyond categorical considerations: A vertex algebra $\V$ and all its modules $\M$ are by definition modules over the Virasoro algebra (that implements conformal symmetry), in particular they are naturally graded $\M=\bigoplus_h \M_h$ by the (possibly generalized) eigenvalues $h$ of $L_0$, called conformal weight, energy or scaling dimension. We typically require the homogeneous components $\M_h$ to be finite-dimensional and the degree bounded from below and discrete, then it makes sense to define the \emph{graded dimension} or (maybe somewhat misleading) \emph{character}\footnote{In general we have graded characters for the action of an element $a$ of the vertex algebra. For $a=1$ this reproduces the graded dimension, the cases below correspond to certain $a\in \h$. In a more general framework, this is what the conformal field theory assigns to a torus, with conformal structure given by $\q$, and to a torus with a puncture decorated by $a$ at $z$. This view suggests or explains the transformation behaviour of the former as a modular form and of the latter as a Jacobi form.}
$$\dimVOA_\M(\q) = \tr(\q^{L_0-\frac{c}{24}})
=\sum_h \dim(\M_h) \q^h,$$
where $\q$ is a formal variable, not to be confused with the root of unity $q$, and $\q^{-\frac{c}{24}}$ is a correction factor depending on the central charge of the Virasoro algebra. If there are additional gradings, these can be taken into account, for example in some cases below we have an action of the abelian Lie algebra $\h=\langle H_1,\ldots, H_n\rangle_\C$, then we consider as character
$$\dimVOA_\M(\q,z) = \tr(\q^{L_0-\frac{c}{24}}\,z^H)
=\sum_{h,\lambda} \dim(\M_{h,\lambda}) \q^h z^\lambda,$$
where $\lambda\in \h^*$ are simultaneous $\h$-eigenvalues and we abbreviate 
$$z^H=z_1^{H_1}\cdots z_n^{H_n}\quad\text{ and }\quad 
z^\lambda=z_1^{\lambda(H_1)}\cdots z_n^{\lambda(H_n)}
=z_1^{(\lambda,\omega_1)}\cdot z_n^{(\lambda,\omega_n)},$$
for $(-,-)$ some nondegenerate inner product and $\omega_1,\ldots,\omega_n$ a basis dual to $H_1,\ldots,H_n$, later the fundamental weights.

\begin{example}
    The Heisenberg vertex algebra $\mathcal{F}^n$ has modules called \emph{Fock modules} $\mathcal{F}_\lambda$ for $\lambda\in\R^n$, which are as graded vector space $\C[a_{-1},a_{-2},\ldots]^{\otimes n}e^\lambda$. Depending on a choice $Q\in \h$ parametrizing different actions of the Virasoro algebra with central charge $c=n-12|Q|^2$, the generator $e^\lambda$ has conformal weight $h=\q^{|\lambda-Q|^2-|Q|^2}$ and the $a_{-k}$ have conformal weight $k$. Accordingly, the character is 
$$\dimVOA_{\mathcal{F}_\lambda}(\q,z) 
=\left(\q^{-\frac{1}{24}}\sum_{k\geq 0}p(k)\q^{k}\right)^n \q^{|\lambda-Q|^2} z^\lambda
=\frac{1}{\eta(\q)^n}  \q^{|\lambda-Q|^2} z^\lambda,$$
    with $p(k)$ the number of partitions of the integer $k$ and $\eta(\q)$ the Dedekind eta function.\\

    The lattice vertex algebra $\V_\Lambda$ is a simple current extension of  $\mathcal{F}^n$ and has modules $\V_\lambda=\bigoplus_{\alpha\in\Lambda} \mathcal{F}_{\alpha+\lambda}$ depending on a coset $\lambda+\Lambda$ in the dual lattice $\Lambda^*$. Accordingly the character  is
$$\dimVOA_{\V_\lambda}(\q,z) 
=\frac{1}{\eta(\q)^n}  \sum_{\alpha\in \Lambda}\q^{|\alpha+\lambda-Q|^2} z^{\alpha+\lambda}
=\frac{\Theta_\lambda(\q,z)}{\eta(\q)^n}$$
    where $\Theta_\lambda(\q,z)$ is a Jacobi theta series. 
\end{example}

The character is not directly multiplicative with respect to the tensor product $\otimes_\V$, since for example $\V$ is the tensor unit. For rational vertex algebras (i.e. category of representations a finite semisimple category), the Verlinde formula  shows that the \emph{quantum dimension}
$$\qdimVOA(\M)=\lim_{\tau\to +\i\infty}\frac{\dimVOA_\M(\q)}{\dimVOA_\V(\q)},\quad \q=e^{2\pi\i\tau}$$
is a well-defined scalar and multiplicative with respect to $\otimes_\V$. For example (with $z=1$, otherwise we obtain a $\h^*$-graded dimension)
$$\qdimVOA(\mathcal{F}_\lambda)
=\lim_{\tau\to +\i\infty} \q^{|\lambda-Q|^2-|Q^2|} 
=1$$
if the inner product is positive definite. Similarly $\qdimVOA(\V_\lambda)=1$.
For more general vertex algebras, such as $\W_p(\g)$, a similar behavior is expected. For vertex algebras with continuous spectra of simple objects, a notion of regularized quantum dimensions is developed and applied to our setting in \cites{CM17,CMW17}. \\

 In the case of the Feigin-Tipunin algebra $\W_p(\g)$, the kernel of screenings in Section \ref{sec_VOA} for $\g$ simply-laced with rank $n$, there is a beautiful method to produce modules (in good cases simple modules, see below) and compute the graded characters. This approach was suggested in \cite{FT10} Section~6 and done in \cite{Sug21} Section~4.4, very useful early computations can be found in \cite{BM17}. It is inspired by the Borel-Weil-Bott theorem for Lie algebras and we shall now describe it briefly: 
 
 Consider the action of the Borel subgroup $B$ on  $\V_\Lambda$ by so-called long screening operators, take the corresponding bundle $\xi_\lambda$ on $G/B$, and take the cohomology $\Hn^0(\xi_\lambda)$. As announced in \cite{FT10} and proven in \cite{Sug21} Main Theorem we have that $\Hn^0(\xi_0)=\W_p(\g)$ and $\Hn^0(\xi_\lambda)=\W_\lambda$ coincides with the module constructed as kernel of screenings, under a certain condition on $\lambda$. Moreover, by \cite{Sug21} Corollary 4.11 the higher cohomologies vanish under the following assumption (conjecturally: always)

 \begin{assumption}\label{ass_smalllambda}
    Assume $\lambda$ fulfills $(\sqrt{p}\bar{\lambda}+\rho,\theta)\leq p$, where $\theta$ is the longest root. 
 \end{assumption}
 
 Assumption \ref{ass_smalllambda} holds for $\bar{\lambda}=0$ if $p \geq h^\vee-1$. Indeed,  the Coxeter number (which is equal to the dual Coxeter number in the simply laced case) is $h=1+\sum_{i} m_i$ if $\theta=\sum_i m_i\alpha_i$ and the dual Weyl vector (which is equal to the Weyl Vector in the simply laced case) has the property $(\rho^\vee,\alpha_i)=1$.
 
 In this case the character can be computed by the Lefschetz- or Atiyah-Bott fixed point formula, where the fixed points of the action of the Cartan group on the flag variety is given by Weyl group elements, see \cite{Sug21} formula (141)-(142) and the previous computations in \cite{BM17}:
\begin{align*}
\dimVOA_{\Hn^k(\xi_\lambda)}(\q,z)
    &=\sum_{k\geq 0} (-1)^k \tr_{\Hn^k(\xi_\lambda)}(\q^{L_0-\frac{c}{24}z^H}z^H) \\
     &=\sum_{w\in W}  \tr_{\mathcal{O}_{G/B}(U_w)\otimes \V_\lambda}(\q^{L_0-\frac{c}{24}z^H}z^H)\\
     &=\frac{1}{\eta(\q)^n}\sum_{\alpha\in Q} 
    \q^{\frac{1}{2}|\sqrt{p}(\alpha+\hat{\lambda}+\rho)+\bar{\lambda}-\frac{1}{\sqrt{p}}\rho|^2}
    \sum_{w\in W} \frac{z^{w.(\alpha+\rho+\hat{\lambda})-\rho}}{\prod_{\beta\in\Phi^{+}(\g)}(1-z^{-\beta})}\\
     &=\sum_{\alpha \in P_+\cap Q} \chi_{L_{\alpha+\hat{\lambda}}^\g}(z)\dimVOA_{T^+_{\sqrt{p}\bar{\lambda},\alpha+\hat{\lambda}}}(\q)
\end{align*}
Here $\Phi^{+}(\g)$ is the set of positive roots and $\lambda=\bar{\lambda}+\hat{\lambda}$ for $\lambda\in\frac{1}{\sqrt{p}}P$ is a presentation of $\lambda$ with the unique $\sqrt{p}P$-coset representative of the form $\bar{\lambda}=\frac{1}{\sqrt{p}}\sum_i s_i\omega_i$ with $1\leq s_i\leq p-1$ and $\hat{\lambda}$.  The second-to-last line without the factor $\frac{1}{\eta(\q)^n}$ can be interpreted as a higher rank false theta function. The last computation describes how the  module $\Hn^0(\xi_\lambda)$ decomposes after restriction into a sum of simple modules $L_{\alpha+\hat{\lambda}}^\g$ module over $\g$ times a simple module defined in \cite{ArF19} over the $\mathrm{W}$-algebra $\mathrm{W}_p(\g)$, which is the Hamiltonian reduction of the affine Lie algebra $\hat{\g}$. This decomposition is the key to many proofs in \cites{Sug21,Sug23}.\\

Then \cite{Sug23} Theorem 1.1 shows for $p\geq h^\vee-1$ that $\W_p(\g)$ is simple as a vertex algebra and in fact all modules $\W_\lambda$ are simple.

\begin{comment}
\cite{Sug21} Lemma 4.21 gives for $p\geq h^\vee$ equivalence between the W-algebra and simple quotient and \cite{Sug23} Theorem 1.2 identifies for $p\geq h^\vee$ and $\lambda$ fulfilling Assumption \ref{ass_smalllambda} the modules $T$ in the character formula with the simple modules of the W-algebra. The main result of this discussion is

\begin{theorem}[\cite{Sug23} Theorem 1.1]
For $p\geq h^\vee-1$ the vertex algebra $\W_p(\g)$ is simple, and for all $\lambda$ fulfilling Assumption \ref{ass_smalllambda} the modules $\W_\lambda$ are simple. 
\end{theorem}
\end{comment}

\begin{comment}
\begin{remark}
\cite{CM17} Corollary 18 computes the regularized quantum dimension
\marginpar{Discuss with TC, also $\mu,\gamma$ typo?}
$$\dimVOA_{\Hn^k(\xi_\lambda)\cap \mathcal{F}_\mu}^\epsilon(q,z)
=e^{-2\pi(\mu+\lambda,\epsilon)}\chi_{L_{-\sqrt{p}\bar{\lambda}}^\g}(\frac{-\i\epsilon}{p})$$
\footnote{
CMW17 Thm 25 says that $\epsilon$-dim is a ring homomorphism for big epsilon and a ring hmomorhpism to minimal model for small epsilon. Semisimplification}
\end{remark}

Interesting contents of Sug-Papers

SUG21 (=20)

Main Theorem: 
H0 = Wp 
H0 = kernel modules under some condition on lambda
H0 decomposes as W0 module

Character formula

WpQ =kernel of screenigs, W0 Kernel of long screeings ,H0=cohomology

bfW is the W-algebra (so W=sum L *  bfW)
bfL is the simple W-alg module
bfT DS reduction 
Thm1.2: p>h, fundamental alcove: bfW=bfL=bfT

SUG23 (=21)

p>h simplicity
Equality of W-algebra modules
KL Formula

\end{comment}

\subsection{Main proof}

\cite{BM17} Theorem 8.1 computes the quantum dimension of these modules in lattice degree $0$, i.e. as modules over the so-called singlet vertex algebra $\W_p(\g)\cap \mathcal{F}_0$. Note the very useful computation in \cite{BM17} Lemma 6.1, analogous to the Weyl dimension formula, which gives the limit $z\to 1$ of the Weyl group sum above.
%$$\sum_{w\in W} \frac{z^{w.(\alpha+\rho+\hat{\lambda})-\rho}}{\prod_{\beta\in\Phi^{+}}(1-z^{-\beta})}\\$$
 Using the modularity of these characters and asymptotic expansions they find in \cite{BM17} Theorem 6.2 that the asympotics $\tau=\i t\to \i\infty$ of the generalized false theta function is 
\begin{align*}
\eta(\q)^n \dimVOA_{\Hn^k(\xi_\lambda)}(\q,z)
&=\frac{\dim_\C(L_{-\sqrt{p}\bar{\lambda}}^\g)}{p^{|\Phi^+(\g)|}}+O(\sqrt{t}).
\end{align*}
From this follows the quantum dimensions
\begin{align}\label{formula_qdimSimple}\dimVOA_{\Hn^k(\xi_\lambda)\cap \mathcal{F}_0}(\q,z)
=\dim_\C(L_{-\sqrt{p}\bar{\lambda}}^\g),
\end{align}
and in particular $\dimVOA_{\Hn^k(\xi_0)\cap \mathcal{F}_0}(\q,z)=1$, as it should be for the tensor unit. On the other hand, it becomes clear that over this vertex algebra the Fock module has quantum dimension 
\begin{align}\label{formula_qdimFock} \dimVOA_{\Hn^k(\xi_\lambda)\cap \mathcal{F}_0}(\q,z)
=p^{|\Phi^+(\g)|}.
\end{align}
Hence we find that, over the singlet algebra and assuming $p>h^\vee-1$, the quantum dimension of $A$ coincides with the dimension of the Nichols algebra $\NicholsOf(M)$ if $q$ is a $2p$-to root of unity and $\g$ is simply laced. In this case,  Theorem~\ref{thm_KLviaFP} reads as follows:

\begin{theorem}\label{thm_KLviaQD}
    Let $\W_p(\g)\subset \V_\Lambda$ be the Feigin-Tipunin algebra for $\g$ simply-laced and $p>h^\vee-1$. Assume that $\RepVOA(\W_p(\g))$ is a braided tensor category, such that (at least for $A$) the quantum dimensions in the vertex algebra sense in Formulas \eqref{formula_qdimSimple} and \eqref{formula_qdimFock} coincides with the Frobenius-Perron  dimension in the categorical sense.     
    Then we have an equivalence  of tensor categories resp. braided tensor categories 
    \begin{align*}
        \UU_A &= \Mod(\Nichols)(\CC) \\
        \UU&=\cZ_\CC(\Mod(\Nichols)),
    \end{align*}
    where $\Nichols\cong u_q(\g)^+$ is the Nichols algebra associated to the diagonal braiding $q_{ij}=q^{(\alpha_i,\alpha_j)}$ for $q=e^{\frac{2\pi\i}{2p}}$. This is to say, $\UU$ is equivalent to representations over a (quasi-) quantum group $\tilde{u}_q(\g)$. 
\end{theorem}

\section{Counterexamples}\label{sec_counterexamples}

We want to mention some situation in which the assertion cannot be true and which have to be kept in mind for the proofs in this article:

\begin{example}\label{exm_biggerThanNichols}
The screening momenta $\alpha_1,\cdots,\alpha_n$ should fulfill a condition we called \emph{supolarity} on the scalar products  $(\alpha_i,\alpha_j)$, not merely their exponentials $q_{ij}$, see \cite{Len21} Section 5.2. It roughly states that the scalar products are not too negative and it ensures the convergence of certain Selberg-type integrals. 

If this is not the case, then the algebra of screening operators $\NicholsScreenings$ can be larger then the Nichols algebra $\Nichols(q)$ with $q_{ij}=e^{\pi\i(\alpha_i,\alpha_j)}$, and in view of Section \ref{sec_NicholsArguments} necessarily infinite-dimensional. In turn it is to be expected that the kernel-of-screening vertex algebra is not $C_2$-cofinite, and one might hope for the category of representation to be related to $\NicholsScreenings$ instead.
\bigskip

As concrete examples consider for the logarithmic Kazhdan-Lusztig correspondence the case $p=1$ for $\g=\sl_3$. We have 
$$(\alpha_i,\alpha_j)=\begin{pmatrix} 2 & -1 \\ -1 & 2\end{pmatrix}$$
Then $\zem_1,\zem_2$ are at the same time long and short screening operators. A quick calculation with the usual vertex algebra OPE commutator formula shows that they generate roughly $U(\sl_3)^+$ in a version with nontrivial symmetric braiding
$$\Nichols^{Scr}=\langle \zem_1,\zem_2 \;\mid\; ([\zem_1,[\zem_1,\zem_2]_+]_-=[\zem_2,[\zem_2,\zem_1]_+]_-=0\rangle$$
However, this is not the Nichols algebra: We have 
$$q_{ij}=e^{\pi\i(\alpha_i,\alpha_j)}=\begin{pmatrix} 1 & -1 \\ -1 & 1\end{pmatrix}$$
and hence the Nichols algebra is the $q$-commutative ring 
$$\NicholsOf(q)=\langle x_1,x_2 \;\mid\; [x_1,x_2]_+=0\rangle$$
It differs by the additional primitive element $[x_1,x_2]_+$, which is central and of infinite order. Accordingly, one can determine the kernel of screenings: it is a large non-simple vertex algebra extension of the Virasoro algebra. 

\bigskip

A similar behavior appears if $p<0$, in which case the truncation relation $(\zem_i)^{p_i}=0$ is violated and instead is equal to the long screening, see \cite{Len21} Section 6.4, and this case also appears hidden in many Lie superalgebras $\sl(n|m)$ for positive $p$, see \cite{FL22}.
\end{example}

On the other hand there is examples that are known to be non-rigid and cannot fully described by the Nichols algebra in the free field realization

\begin{example}
The vertex algebra $\W_{2,3}$ is an extension of the Virasoro algebra at central charge $c=0$ and is the kernel of a single screening operator $\zem_\alpha$ with $(\alpha,\alpha)=4/3$. It
is studied in our context in \cites{RP08,GRW09} (with some parts conjectural): It has $13$ simple modules. The tensor unit is an indecomposable extension 
$$\W(2)\,\to \W\to \W(0),$$
hence the dual $\W^*$ is a non-isomorphic indecomposable extension. Some notable tensor products are
$$\W(0)\boxtimes_\W \W(0)=\W(0),\qquad
\W(0)\boxtimes_\W \W(h)=0,h\neq 0,$$
and in particular the tensor product is not exact. The Grothendieck-ring obtained by setting the class $[\W(0)]=0$ corresponds conjecturally to a quantum group $u_q(\sl_2)$. 

If we try to apply the method in the present article, we discover that in this case the induction functor $A\otimes(-)$ is not faithful anymore, as it sends $\W(0)$ to the zero object. It would be reasonable to expect, that our results then prove precisely the conjecture in the previous paragraph.
\bigskip

In principle, all $\W_{p,q}$ should be the rank $1$ building blocks of a general kernel-of-screening operator algebra, but for the reason discussed we have to restrict ourselves to those corresponding to $\W_p=\W_{p,1}$, whence the condition $(\alpha,\alpha)=2/p$ in our main theorem. One would expect that the quantum groups corresponding to $p/q$ and $q/p$ both appear in such a setting, as in the original Kazhdan-Lusztig correspondence. 
\end{example}

Let us also mention that it would be very interesting to study Hopf algebras where the coradical is not a Hopf subalgebra.

\end{document}